\documentclass[11pt,a4paper]{amsart}
\usepackage[utf8]{inputenc}
\usepackage{amsfonts, amssymb, amsmath, amsthm}
\usepackage{mathrsfs}
\usepackage[left=2.7cm,right=2.7cm,top=3.5cm,bottom=3cm]{geometry}
\usepackage[T1]{fontenc}
\usepackage{graphicx}
\usepackage{color}
\usepackage{xcolor}
\usepackage{multicol} 
\usepackage{stmaryrd} 
\usepackage{tikz-cd}
\usepackage{subcaption}
\usepackage{enumerate} 
\usepackage{hyperref}
\usepackage[toc,page]{appendix}
\usepackage{comment}

\renewenvironment{thebibliography}[1]{%
  \begin{oldthebibliography}{#1}%
  \normalsize
}{\end{oldthebibliography}}

\title[Quaternionic big Heegner points over totally real fields]%
      {Quaternionic big Heegner points\\ over totally real fields}

\author{Ignacio M. Jiménez}
\date{}
\address{Dipartimento di Matematica, Universit\`a di Genova\\ via Dodecaneso 35, 16146, Genova, Italy}
\email{ignacio.munoz.jimenez@edu.unige.it}

\renewcommand{\frak}{\mathfrak}
\renewcommand{\hat}{\widehat}
\renewcommand{\tilde}{\widetilde}
\renewcommand{\t}[1]{\text{#1}}
\renewcommand{\bf}[1]{\textbf{#1}}

\renewcommand{\mod}[1]{ (\t{mod} #1)}
\newcommand{\bb}[1]{\mathbb{#1}}
\newcommand{\mcal}[1]{\mathcal{#1}}
\newcommand{\LL}{\left}
\newcommand{\RR}{\right}

\newcommand{\Hom}{\text{\normalfont{Hom}}}

\newcommand{\Gal}{\text{\normalfont{Gal}}}

\newcommand{\<}[1]{\langle #1 \rangle}

\newcommand{\Div}{\text{\normalfont{Div}}}

\newcommand{\longtwoheadrightarrow}{\relbar\joinrel\twoheadrightarrow}

\theoremstyle{definition}
\newtheorem{definition}{Definition}[section]
\theoremstyle{definition}

\theoremstyle{definition}
\newtheorem{remark}{Remark}[section]
\theoremstyle{definition}

\theoremstyle{definition}
\newtheorem{proposition}{Proposition}[section]
\theoremstyle{definition}
\newtheorem{theorem}{Theorem}[section]
\theoremstyle{definition}
\newtheorem{lemma}{Lemma}[section]
\theoremstyle{definition}
\newtheorem{corollary}{Corollary}[section]
\theoremstyle{definition}
\newtheorem{assumption}{Assumption}[section]
\theoremstyle{definition}

\theoremstyle{definition}

\theoremstyle{definition}
\newtheorem{conjecture}{Conjecture}[section]

\newtheorem{introthm}{Theorem}


\begin{document}

\begin{abstract}
    In this work, we extend Howard's construction of compatible families of Heegner points to the setting of towers of Gross curves and Shimura curves over totally real fields. Following the strategy of Longo and Vigni, our approach simultaneously treats totally definite and indefinite quaternion algebras. We then extend their interpolation methods to define big Heegner points attached to families of Hilbert modular forms of parallel weight under the weak Heegner hypothesis. Applying this construction, we build in the definite setting a totally real analogue of Longo--Vigni’s two-variable \(p\)-adic \(L\)-function, and in the indefinite setting, a system of big Heegner classes in the sense of Howard.
\end{abstract}

\maketitle

\tableofcontents

\section{Introduction} 

The goal of this work is to generalize to the totally real setting the constructions and results of \cite{LV10}. The first new contribution is an explicit version of their construction of a compatible family of Heegner points in towers, inspired by \cite{CL16}, which we extend beyond the rational case to totally real fields and for more general conductor. Building on this, we define big Heegner points attached to the Hida family of a Hilbert modular form of parallel weight under the \textit{weak} Heegner hypothesis.\\

In the remainder of the paper, we study arithmetic applications in both the (totally) definite and indefinite quaternionic settings. In the definite case, we construct a two-variable \(p\)-adic \(L\)-function, conjectured to interpolate the values of Rankin--Selberg \(L\)-functions corresponding to different specializations of the Hida family. This construction allows us to formulate an Iwasawa main conjecture in this setting. In the indefinite case, our main contribution is the construction of big Heegner classes via a big Kummer map, following the strategy of \cite{How06} and providing an alternative approach to those of \cite{Fou12} and \cite{Dis22}.

\subsection{Motivation and context}
Kolyvagin's strategy for studying Selmer groups of elliptic curves via Heegner points remains one of the major advances toward the Birch--Swinnerton--Dyer conjecture. Howard~\cite{How06} later extended this approach by introducing \textit{big Heegner points} to study the arithmetic of Hida families. These points interpolate the classical weight-two Heegner points and give rise to Euler systems for higher-weight ordinary modular forms. The construction was subsequently adapted to the rational quaternionic setting by Longo and Vigni~\cite{LV10}.\\

In their seminal work~\cite{BD96}, Bertolini and Darmon developed the quaternionic framework, replacing modular curves with Shimura curves, thereby relaxing the Heegner hypothesis and establishing a Gross--Zagier-type relation between Heegner points and special values of \(L\)-functions. Depending on the conductor of the elliptic curve, they distinguished two cases: the \textit{indefinite} and the \textit{definite} settings. In the former, Heegner points reflect information about the derivative of the Rankin--Selberg \(L\)-function; in the latter, they encode its special values. Geometrically, indefinite Shimura curves resemble modular curves, while definite ones (or \textit{Gross curves}) exhibit a fundamentally different behavior.\\

Following Longo and Vigni, this project extends Bertolini--Darmon's ideas to Hida families and seeks an analogue of their results in the totally real setting.

\subsection{Set-up}
Let \(F\) be a totally real field of degree \(d:=[F:\bb{Q}]\) and \(\frak{n}\) an ideal in \(\mcal{O}_F\). Fix a rational prime \(p>3\) coprime with \(\frak{n}\) and unramified in \(F\). Consider \(K/F\) a CM field of relative discriminant \(d_{K/F}\) coprime with \(\frak{n}p\). The field \(K\) determines a factorization
\[\frak{n}=\frak{n}^+ \frak{n}^-,\]
where the primes \(v|\frak{n}^+\) (resp. \(\frak{n}^-\)) are split (resp. inert) in \(K\). Assume \(\frak{n}^-\) squarefree (weak Heegner hypothesis) and denote \(r\) the number of prime factors dividing \(\frak{n}^-\).
We distinguish two scenarios:
\begin{itemize}
    \item The \textbf{(totally) definite} setting when \(r+d\equiv 0 \mod{2}\).
    \item The \textbf{indefinite} setting when \(r+d\equiv 1 \mod{2}\).
\end{itemize}
\vspace{3mm}

Fix an embedding \(\tau: F\hookrightarrow \bb{R}\). Consider \(B/F\) the unique quaternion algebra of discriminant \(\frak{n}^-\) such that, for all archimedean primes \(\sigma\neq \tau\), we have \(B\otimes_{F,\sigma} \bb{R}=\bb{H}\), where \(\bb{H}\) denotes Hamilton's quaternions. Notice we are in the definite setting when \(B\otimes_{F,\tau}\bb{R}=\bb{H}\) and in the indefinite setting when \(B\otimes_{F,\tau}\bb{R}=\t{M}_2(\bb{R})\), as it happens over \(\bb{Q}\).\\

Consider, for each \(m\geq 1\), an open compact subgroup \(U_{0,1}(\frak{n}^+,p^m)\) regarding the classical \(\Gamma_{0,1}(N,p^m)\) (see Section \ref{Review of Hilbert modular forms and Hida theory} for precise definitions). We define the double cosets
\[X_m^{(K)} := B^\times \backslash (\Hom_{F,\tau}(K,B)\times \hat{B}^\times) \slash U_{0,1}(\frak{n}^+,p^m),\]
and we define an action of the absolute Galois group \(G_K\)
\[[(f,g)]^\sigma := [(f,\hat{f}(x_\sigma)g)],\]
which mimics Shimura's reciprocity law, where \(\hat{f}\) denotes the adelization of \(f\) and \(\t{rec}_K(x_\sigma)=x\). Define also the natural projections
\[\alpha_m: X_m^{(K)} \longrightarrow X_{m-1}^{(K)}.\]

Choose an ideal \(\frak{c}\subset\mcal{O}_K\). Denote by \(H_\frak{c}\) the ring class field of conductor \(\frak{c}\) relative to \(K/F\). Denote \(F[m]\) the ray class field of modulus \(p^m\infty\) over \(F\) and, for a field \(A/F\), denote \(A[m]\) the field composition \(A\cdot F[m]\).

\subsection{Main results}

In this setting, we construct an explicit compatible family of Heegner points 
\(\{P_{\frak{c},m}\}_{\frak{c},m}\) in \(\{X_m^{(K)}\}_{m\geq 1}\) such that, for each ideal 
\(\frak{c}\subset \mcal{O}_F\) coprime to \(\frak{n}d_{K/F}\) and supported only at primes 
\(v\) inert in \(K\), one has
\[
\{P_{\frak{c},m}\}_\frak{c} \subset X_m^{(K)},
\]
and the following compatibility relations are satisfied:
\begin{enumerate}
    \item \textbf{Horizontal compatibility:}
    \begin{itemize}
        \item \(U_p(P_{\frak{c}p^{n-1},m}) = 
        \operatorname{tr}_{H_{\frak{c}p^{m+n}}[m+n]/H_{\frak{c}p^{m+n-1}}[m+n]} 
        (P_{\frak{c}p^n, m})\).
        \item \(T_v(P_{\frak{c},m}) = 
        \operatorname{tr}_{H_{\frak{c}v p^m}[m]/H_{\frak{c}p^m}[m]} 
        (P_{\frak{c}v, m})\).
    \end{itemize}
    \vspace{1mm}
    \item \textbf{Vertical compatibility:}
    \begin{itemize}
        \item \(U_p(P_{\frak{c},m-1}) = 
        \alpha_m \big(\operatorname{tr}_{H_{\frak{c}p^m}[m]/H_{\frak{c}p^{m-1}}[m]} 
        (P_{\frak{c},m})\big)\).
    \end{itemize}
    \vspace{1mm}
    \item \textbf{Galois compatibility:}
    \begin{itemize}
        \item \(P_{\frak{c},m}^\sigma = 
        \langle \vartheta(\sigma) \rangle P_{\frak{c},m}\).
    \end{itemize}
\end{enumerate}

Here \(v\nmid \frak{ncp}\) is a prime of \(\mcal{O}_F\) inert in \(K\). The operators 
\(U_p\), \(T_v\), and \(\langle \cdot \rangle\) are Hecke operators acting via Hecke 
correspondences on \(X_m^{(K)}\). See Section~\ref{Construction of big Heegner points} 
for precise definitions.\\

Next, we use this family to construct \textit{big Heegner points} attached to a Hida family of Hilbert modular forms. Let \(k\) be a positive integer, and fix a \(p\)-satabilized ordinary Hilbert cuspform \(f\) of parallel weight \(k\), trivial character and level \(\frak{n}\). By Hida's theory, there exists a big Galois representation associated with \(f\): a complete local Noetherian ring \(\mcal{R}\) and a rank-two \(\mcal{R}\)-module \(\mathbf{T}\) whose specializations recover the Galois representations attached to classical forms. Extending the ideas of Howard~\cite{How06}, one considers a twist \(\mathbf{T}^\dagger\) admitting a perfect alternating pairing
\[
\mathbf{T}^\dagger \times \mathbf{T}^\dagger \longrightarrow \mcal{R}(1).
\]

Let \(\mathbb{D}_m\) be the group of divisors on \(X_m^{(K)}\), and set \(\mathbb{D} := \varprojlim_m \mathbb{D}_m\)
under the projection maps. In Section \ref{Construction of big Heegner points}, we refine these into modules 
\[\mathbf{D}_m^\dagger \t{\hspace{0.5cm} and \hspace{0.5cm}} \mathbf{D}^\dagger := \varprojlim_m \mathbf{D}_m^\dagger,\]
endowed with an action of Hida’s big ordinary Hecke algebra \(\frak{h}_\infty^{\mathrm{ord}}\). The projected family of points previously built satisfies a horizontal compatibility relation of Kolyvagin type. This is the content of Theorem \ref{Main theorem A} below.
\begin{introthm}
    There is a family of points \(\{\mcal{P}_\frak{c}\}_\frak{c}\subset \bf{D}^\dag\), for \(\frak{c}\) conductor coprime with \(\frak{n}d_{K/F}\) and whose prime divisors are inert in \(K\), such that
    \[\mcal{P}_\frak{c} \in H^0(\Gal(K^\t{ab}/H_\frak{c}),\bf{D}^\dag),\]
    and satisfying the following horizontal compatibility relations:
    \begin{align*}
        U_p(\mcal{P}_\frak{c}) &= \t{cor}_{H_{\frak{c}p}/H_\frak{c}}(\mcal{P}_{\frak{c}p}),\\
        T_v(\mcal{P}_\frak{c}) &= \t{cor}_{H_{\frak{c}v}/H_\frak{c}}(\mcal{P}_{\frak{c}v}),
    \end{align*}
    for each \(v\nmid \frak{nc}p\) inert.
\end{introthm}\vspace{2mm}

Up to this point, the results apply to both the definite and indefinite quaternionic settings, that is, regardless of the parity of the number of primes dividing \(\frak{n}^-\). From now on, given the fundamentally different nature of the underlying objects, we study each setting separately.\\

This distinction reflects the sign of the functional equation (the global root number) of the Rankin--Selberg \(L\)-function attached to the Hida family and the quadratic extension \(K/F\). In the totally definite case the root number is \(+1\), leading to nontrivial central \(L\)-values and hence to \(p\)-adic \(L\)-functions, while in the indefinite case it is \(-1\), forcing vanishing at the center and motivating the construction of Heegner classes instead.

\subsubsection{Totally definite setting}

In the totally definite setting, the natural objects to consider are \emph{Gross curves}, as reviewed in Section~\ref{Gross curves and Shimura curves over totally real  fields}.  
Let \(K_\infty/K\) denote the \(p\)-anticyclotomic \(\mathbb{Z}_p\)-extension.  

Following \cite{LV10}, we construct a \emph{big theta element}
\[
\theta_\infty \in \mathcal{R}[[\mathrm{Gal}(K_\infty/K)]] =: \mathcal{R}_\infty,
\]
and define a corresponding \emph{two-variable \(p\)-adic \(L\)-function}
\[
\mathcal{L}_p(f/K) = \theta_\infty \cdot \theta_\infty^*.
\]
This function is conjectured to interpolate (up to explicit periods and local factors) the special values of the Rankin--Selberg \(L\)-functions associated with the data \((\chi,\frak{p})\), that is, a finite-order character \(\chi \in \mathrm{Gal}(K_\infty/K)\) and arithmetic prime \(\frak{p}\) of \(\mathcal{R}\). Denote the weight of the \(\frak{p}\)-specialization as \(k_\frak{p}\).

We formulate two conjectures in this setting.

\begin{conjecture}[Non-vanishing conjecture]
\[
\mathcal{L}_p(f/K;\chi,\mathfrak{p}) \neq 0 
\;\Longleftrightarrow\;
L(f_\mathfrak{p},\chi,k_\mathfrak{p}/2) \neq 0.
\]
\end{conjecture}

This conjecture predicts a precise correspondence between the non-vanishing of the algebraic \(p\)-adic interpolation and the non-vanishing of the complex \(L\)-values. The weight 2 specializations are the totally real counterpart of \cite{BD96}, while the higher weight specializations should be compared to the totally real counterpart of the theta elements of \cite{CH16}, as proved in \cite{CL16} over \(\bb{Q}\).\\

Our second conjecture is an analogue of the two-variable Iwasawa Main Conjecture for the totally definite setting.

\begin{conjecture}[Two-variable Iwasawa Main Conjecture]
\[
(\mathcal{L}_p(f/K)) 
= 
\mathrm{Char}_{\mathcal{R}_\infty}
\Big(
\widetilde{H}^1_{f,\mathrm{Iw}}(K_\infty,\mathbf{A}^\dagger)^\vee
\Big),
\]
where \(\mathrm{Char}_{\mathcal{R}_\infty}\) denotes the characteristic ideal in the Iwasawa algebra \(\mathcal{R}_\infty\).
\end{conjecture}
See later Section \ref{Anticyclotomic Iwasawa main conjecture in the totally definite setting} for precise definitions. This statement is the totally real counterpart of \cite[Conjecture 9.12]{LV10}.

\subsubsection{Indefinite setting} 
The geometric objects considered in this part are (indefinite) Shimura curves over totally real fields. We propose a construction of a compatible family of Heegner classes. Our approach, distinct from those of \cite{Fou12} and \cite{Dis22}, is closer in spirit to the original method of \cite{How06}, where Howard constructs his classes using a \emph{big Kummer map}.\\

Let \(H_\frak{c}^{(\frak{n}p)}\) be the maximal abelian extension of \(H_\frak{c}\) outside \(\frak{n}p\). This is our main result in this setting (see Theorem \ref{Main Theorem 2} below):
\begin{introthm}
    There exists a family of classes \(\{\kappa_\frak{c}\}_\frak{c}\) such that
    \[\kappa_\frak{c} \in H^1(\Gal(H_\frak{c}^{(\frak{n}p)}/H_\frak{c}),\bf{T}^\dag),\]
    and, for each \(v\nmid \frak{cn}p\) inert, satisfy the horizontal compatibility relations
    \begin{align*}
        U_p(\kappa_\frak{c})&= \t{cor}_{H_{\frak{c}p}/H_\frak{c}}(\kappa_{\frak{c}p}),\\
        T_v(\kappa_\frak{c})&= \t{cor}_{H_{\frak{c}v}/H_\frak{c}}(\kappa_{\frak{c}v}).
    \end{align*} 
\end{introthm}

The specialization of these classes should be compared with generalized Heegner classes. See \cite{Cas19} and \cite{Ota20} for the classes of Howard, and \cite{LMW25} for the classes of Longo and Vigni.

\subsection{General notation}
Fix an embedding \(\overline{\bb{Q}}\hookrightarrow \overline{\bb{Q}}_p\). For a field \(K\), denote \(\bb{A}_K\) its ring of adeles and \(\hat{K}\) its finite part. Let \(F/\bb{Q}\) be a totally real field and fix an embedding \(\tau: F\hookrightarrow \bb{R}\). Consider a ring \(A\) such that \(\tau\) embeds \(\mcal{O}_F\) into it. The letter \(v\) will always represented a finite prime in \(F\). Denote \(A_v := A\otimes_{\mcal{O}_F,\tau} \mcal{O}_{F,v}\). On the other hand, for \(p\) prime in  \(\bb{Q}\), denote
\[ A_p:= A \otimes_\bb{Z} \bb{Z}_p = \prod_{\frak{p}|p} A \otimes_{\mcal{O}_F,\tau} \mcal{O}_{F,\frak{p}}.\]
All the Galois cohomology groups considered below are continuous cohomology groups.

\subsection{Acknowledgments} 
This paper forms part of my PhD thesis at the Università di Genova. I thank my supervisor, Stefano Vigni, for suggesting this project and for his constant support. I am also thankful to Enrico Da Ronche, Luca Mastella, Beatrice Ostorero Vinci and Francesco Zerman for many helpful discussions.

\section{Review of Hilbert modular forms and Hida theory} \label{Review of Hilbert modular forms and Hida theory}

In this section, we recall the definition of Hilbert modular forms that will be used throughout the paper and summarize their main properties. We then review the basic theory of Hida families of Hilbert modular forms, before turning to Nekovar’s extended Selmer groups in the context of a Hida family. Our main references in the totally real setting are \cite{Nek06}, \cite{Fou12}, and \cite{Hid94}. We follow the structure of \cite{ACR25}. For closely related treatments of Hida theory in the rational case, see \cite{LV10} and \cite{How06}.

\subsection{Notation}
Let \(F\) be a totally real field of dimension \(d\). Consider \(I_F\) to be the set of embeddings of \(F\) in \(\bb{R}\). Let \(\mcal{H}\) be the upper-half complex plane and let \(\mcal{H}^{I_F}\) denote the direct sum of \(d\) copies of \(\mcal{H}\). Write \(\textbf{i}:=(i,\dots,i)\in\mcal{H}^{I_F}\), where \(i=\sqrt{-1}\).\\

Consider \(G:=\t{Res}_{F/\bb{Q}} (\t{GL}_{2,F})\). One can identify
\[G(\bb{R}) \simeq  \prod_{\sigma\in I_F} \t{GL}_2(\bb{R}).\]
Denote by \(G^+\) the elements of \(G\) with totally positive determinant. The real points \(G^+(\bb{R})\) act on \(\mcal{H}^{I_F}\) via fractional linear transformation. This means, for
\[\gamma = \begin{pmatrix} 
    a_\sigma & b_\sigma \\ 
    c_\sigma & d_\sigma \\ 
\end{pmatrix}_{\sigma\in I_F}  \hspace{5mm} \t{and} \hspace{5mm} \tau= (\tau_\sigma)_{\sigma\in I_F}\in \mcal{H}^{I_F},
\]
that the action of \(\gamma\) on \(\tau\) is given by
\[ \gamma \cdot \tau := \LL(\frac{a_\sigma \tau_\sigma + b_\sigma}{c_\sigma \tau_\sigma + d_\sigma}\RR)_{\sigma_\in I_F} \in \mcal{H}^{I_F}.\]

\subsection{Definition of Hilbert cuspform} 
Choose the level \(\frak{n}\) to be an integral ideal in \(\mcal{O}_F\). Consider, in \(G(\hat{\bb{Q}})\), the subgroups
\begin{align*}
    U_0(\frak{n}) &= \LL\{ \begin{pmatrix} 
        a & b \\ 
        c & d \\ 
    \end{pmatrix} \in G(\hat{\bb{Z}})  \LL| c \equiv 0 \mod{\frak{n}\hat{\mcal{O}}_F} \RR. \RR\},\\
    U_1(\frak{n}) &= \LL\{ \begin{pmatrix} 
        a & b \\ 
        c & d \\ 
    \end{pmatrix} \in G(\hat{\bb{Z}})  \LL| c,d-1\equiv 0 \mod{\frak{n}\hat{\mcal{O}}_F} \RR. \RR\}.
\end{align*}
Complete the groups by adding an archimedean part 
\(K_\infty\), chosen to be the stabilizer of \(\bf{i}\) in \(G(\bb{R})^+\).\\

In the next definition, let \(U(\frak{n})\) be either \(U_0(\frak{n})\) or \(U_1(\frak{n})\). Let \(k\) be a positive even integer.

\begin{definition}[Hilbert cuspform]
An \textbf{(adelic) Hilbert cuspform} of level \(U(\frak{n})\) and parallel weight \(k\) is a function \(f:G(\bb{A}_F)\rightarrow \bb{C}\) satisfying

\begin{itemize}
    \item For all \(\alpha\in G(\bb{Q})\) and \(v=u\cdot u_\infty\in V(\frak{n})\), 
    \[f(\alpha x v) = f(x) j_k(u_\infty,\bf{i})^{-1},\] 
    where \(j_k\) is the automorphy factor defined by
    \[j_k(\gamma,\tau) := \prod_{\sigma \in I_F} \frac{(c_\sigma \tau_\sigma + d_\sigma)^k}{\det(\gamma_\sigma)},\]
    for \(\gamma= \begin{pmatrix} 
        a & b \\ 
        c & d \\ 
    \end{pmatrix} \in G(\bb{R})\) and \(\tau = (\tau_\sigma)_{\sigma\in I_F}\in \mcal{H}^{I_F}\).\\

    \item For every \(x\in G(\hat{\bb{Q}})\) take \(u_\infty \in G(\bb{R})^+\) such that \(u_\infty \bf{i} = z\). We ask the function \(f_x: \mcal{H}^{I_F} \rightarrow \bb{C}\) defined by \(f_x(z):=f(xu_\infty)j_k(u_\infty,\bf{i})\) to be holomorphic.\\
    
    \item For all adelic points \(x\in G(\bb{A}_F)\) and for all additive measures on \(F\backslash \bb{A}_F\) we have
    \[\int_{F\backslash\bb{A}_F} f\LL(\begin{pmatrix} 
        1 & a \\ 
        0 & 1 \\ 
    \end{pmatrix}x\RR)da =0.\]

    \item When \(F=\bb{Q}\), for every \(x\in G(\hat{\bb{Q}})\) the function \(|\t{Im}(z)^{k/2}f_x(z)|\) is uniformly bounded on \(\mcal{H}\).
\end{itemize} 
\end{definition}

Denote by \(S_k(\frak{n};\bb{C})\) the \(\bb{C}\)-vector space of Hilbert cuspforms of level \(U_1(\frak{n})\) and parallel weight \(k\). One can define the \(q\)-expansion of a Hilbert modular form (see for example \cite[Section 2.3.3]{Hid06}). For a subring \(R\subset \bb{C}\) let \(S_k(\frak{n};R)\) denote the \(R\)-module of cuspforms with coefficients in \(R\). 

\subsection{Hecke operators}
We give the definition of the Hecke operators we are interested in using double coset operators.\\

For any \(g\in G\), there is a decomposition as a finite disjoint union of the double coset 
\[U_1(\frak{n})gU_1(\frak{n}) = \bigsqcup_i \gamma_i U_1(\frak{n}).\]
This lets us define the double coset operator
\[([U_1(\frak{n})gU_1(\frak{n})]f)(x) := \sum_i f(x\gamma_i).\]

Like in the classical theory, we will be interested in some particular instances of this construction. These are the diamond operators \(\<{z}\), for \(z\in Z(\hat{\bb{Q}})\), and the operators \(T_\frak{q}\), for all finite primes \(v\nmid \frak{n}\). We refer to this set as the Hecke operators. They provide as with the following well-known definition.
\begin{definition}
    A cuspform \(g\in S_k(\frak{n};\bb{C})\) is called \bf{eigenform} if it is an eigenvector for the Hecke operators.
\end{definition}

We proceed to give a precise definition of the Hecke operators.

\subsubsection{Diamond operators}
For \(g\in \bb{A}_F^\times\), define
\[(\<{g}f)(x):=\LL(\LL[U_1(\frak{n}) \begin{pmatrix} 
        g & 0 \\ 
        0 & g \\ 
\end{pmatrix} U_1(\frak{n})\RR]f\RR)(x)=f(xg).\]

The action of the diamond operators factor through the \textit{narrow class group} modulo \(\frak{n}\)
\[\t{Cl}_F^+(\frak{n}):=F^\times \backslash \bb{A}_F^\times / (F\otimes_\bb{Q} \bb{R})^\times_+ (1+\hat{\frak{n}}).\]
This induces a decomposition of \(S_k(\frak{n};\bb{C})\) in terms of characters
\[\chi:  \t{Cl}_F^+(\frak{n}) \rightarrow \bb{C}^\times,\]
satisfying the parity condition
\[\chi_v(-1) = (-1)^k\]
for all archimedean primes \(v\). We denote by \(S_k(\frak{n},\chi;\bb{C})\) the submodule of \(S_k(\frak{n};\bb{C})\) where \(\<{g}\) acts by multiplication times \(\chi(g)\) for all \(g\in \bb{A}_F^\times\). The character \(\chi\) is called the \textit{nebentype} of \(f\).

\subsubsection{Operators \(T_v\)}
For a prime ideal \(v\subset\mcal{O}_F\), fix a choice of uniformizer \(\varpi_v\). When no confusion can arise, we also use the same symbol \(\varpi_v\) to denote the idele whose \(v\)-component is \(\varpi_v\) and whose components at all other places are equal to 1. Define
\[T(\varpi_v) := \LL[U_1(\frak{n}) \begin{pmatrix} 
    \varpi_v & 0 \\ 
    0 & 1 \\ 
\end{pmatrix} U_1(\frak{n})\RR].\]
The operator will not depend on the choice of the uniformizer (see \cite[Section 2.2]{ACR25}). For this reason, denote \(T(\varpi_v)\) by \(T_v\).

\subsection{Galois representation attached to a Hilbert modular form}

Let \(f\in S_k(\frak{n},\chi;\bb{C})\) be an eigenform. Thanks to the work of Eichler--Shimura, Caroyol and Taylor among others, for a sufficiently large finite field extension \(L/\bb{Q}_p\), there is a \(p\)-adic Galois representation \(\rho_f\) attached to \(f\)
\[\rho_f: G_F \longrightarrow \t{GL}_2(L),\]
characterized by
 \[\t{det}(1-\rho_f(\t{Fr}_v)X) = 1- a_v(f) X + \chi(v) |v|^{k-1} X^2,\]
for all prime ideals such that\(v\nmid p\frak{n}\). Here, \(\t{Fr}\) denotes the \textit{geometric} Frobenius and \(a_v(f)\) the eigenvalue of \(T_v\) acting on \(f\). 

The dual representation \(V(f)^\vee\) satisfies the relation
\[ V(f)^\vee \simeq V(f)(1-k) \otimes \chi^{-1}.\]
For more details on this Galois representation, see \cite[Section 12.4]{Nek06}.

\subsection{Hida theory for Hilbert modular forms} \label{Hida theory}
In the context of Hida theory, the notion of \textit{ordinary} cuspform is fundamental.

For the fixed rational prime \(p\), denote by \(U_p\) the operator
\[U_p := \prod_{\frak{p}|p} T_\frak{p}.\]
\begin{definition}
    We say that that an eigenform \(f\in S_k(\frak{n};\bb{C})\) is \bf{\(p\)-ordinary} if the eigenvalue \(\lambda_p\) of \(U_p\) at \(f\) is a \(p\)-adic unit. Equivalently, if for every \(v|p\), the eigenvalues \(\lambda_v\) of \(T_v\) at \(f\) are \(p\)-adic units.
\end{definition}

Consider the open compact subgroups 
\[U_{0,1}(\frak{n},p^m):= U_0(\frak{n})\cap U_1(p^m).\]
Let \(\mcal{O}\) be the ring of integers of \(L/\bb{Q}_p\), the fixed field of coefficients of the Galois representation attached to \(f\). Consider the cuspforms \(S_k(U_{0,1}(\frak{n},p^m);\mcal{O})\) and the corresponding Hecke algebra \(\frak{h}_k(U_{0,1}(\frak{n},p^m);\mcal{O})\), that is, the \(\mcal{O}\)-algebra generated by the Hecke operators acting on \(S_k(U_{0,1}(\frak{n},p^m);\mcal{O})\). Consider the inverse limit
\[\frak{h}(\frak{n}p^\infty;\mcal{O}) := \varprojlim_m \frak{h}_k(U_{0,1}(\frak{n},p^m);\mcal{O}).\]
Hida (\cite{Hid94}) proved that this algebra does not depend on the weight \(k\). For this reason, it is not explicitly stated in the notation. The compact ring \(\frak{h}(\frak{n}p^\infty;\mcal{O})\) has a direct summand 
\[\frak{h}^\t{ord}(\frak{n}p^\infty;\mcal{O})\] 
on which the operator \(U_p\) is invertible. It is obtained using \textit{Hida's big ordinary projector} \(e= \lim_{r\to \infty} U_p^{r!}\) (see \cite[Section 12.7]{Nek06}). When there is no confusion, denote the algebra \(\frak{h}(\frak{n}p^\infty;\mcal{O})\) as \(\frak{h}\), and similarly the algebra \(\frak{h}^\t{ord}(\frak{n}p^\infty;\mcal{O})\) as \(\frak{h}^\t{ord}\).\\

Denote \(\t{Cl}_F^+(\frak{n}p^\infty):=\varprojlim_m \t{Cl}_F^+(\frak{n}p^m)\). The module \(\frak{h}\) has a natural action of the Iwasawa algebra
\[\underline{\Lambda}:=\varprojlim_m\mcal{O}[\t{Cl}_F^+(\frak{n}p^m)] = \mcal{O}[[\t{Cl}_F^+(\frak{n}p^\infty)]],\]
which arises from the action of the diamond operators. In particular, the action of its group elements on \(\frak{h}_2^\t{ord}\) is given by \(\sigma\mapsto \<{\epsilon(\sigma)}\), for a choice of a natural continuous homomorphism
\[ \epsilon: \t{Cl}_F^+(\frak{n}p^\infty)\longrightarrow \mcal{O}^\times_{F,p}.\]
Similarly, denoting by \(\Gamma\) the free part of \(\t{Cl}_F^+(\frak{n}p^\infty)\), define the corresponding subalgebra \(\Lambda\subset \underline{\Lambda}\) by
\[\Lambda:=\mcal{O}[[\Gamma]].\]

It is a result of Hida that \(\frak{h}^\t{ord}\) is finite and torsion-free as \(\Lambda\)-algebra (see \cite[Section 3]{Hid94}).
 
\begin{definition}
    A \textbf{Hida family} is a \(\Lambda\)-algebra homomorphism 
    \[\mathscr{G}: \frak{h}^\t{ord} \longrightarrow \bb{I},\]
    for \(\bb{I}\) a choice of a \textit{branch}, that is, an integral domain, finite and flat over \(\Lambda\).
\end{definition}

The kernel is a minimal prime which determines a local noetherian domain \(\mcal{R}\), 
\[ \mcal{R}:=\frak{h}^\t{ord}/\t{ker}(\mathscr{G}),\]
which is finite and flat over \(\Lambda\). This ring will be crucial in later constructions.\\

Consider the \(\mcal{O}\)-algebra homomorphism 
\begin{align*}
    P_{k,\chi}: \Lambda &\longrightarrow \mcal{O}\\
    [z] &\longmapsto \chi(z) N_{F_p/\bb{Q}_p}(z)^{k-2},
\end{align*}
where \([\cdot]\) denotes group elements, \(k\) is a non-negative integer and \(\chi: \Gamma \rightarrow \mcal{O}^\times\) a finite character.

\begin{definition}
    An \textbf{arithmetic map} is any element of \(\Hom_{\mcal{O}\t{-alg}}(\bb{I},\overline{\bb{Q}}_p)\) that, restricted to \(\Lambda\), coincides with \(P_{k,\chi}\) for some parallel weight \(k\) and some character \(\chi\). 
\end{definition}

The choice of a Hida family \(\mathscr{G}\) and an arithmetic map \(P_{k,\chi}\) determines an ordinary cuspform. In fact, denote by \(\mcal{P}\) the kernel in \(\frak{h}^\t{ord}\) of the composition \(P_{k,\chi} \circ \mathscr{G}\). Then the induced map
\[ \frak{h}^\t{ord}/\mcal{P} \xrightarrow{\ P_{k,\chi} \circ \mathscr{G}\ } \overline{\bb{Q}}_p  \]
realizes the system of Hecke eigenvalues of a canonical ordinary normalized \(p\)-stabilized eigenform. Conversely, given an ordinary cuspidal Hilbert eigenform \(f\) of parallel weight, 
consider the map
\[
  \lambda_f : \frak{h}^\t{ord} \longrightarrow \overline{\bb{Q}}_p
\]
sending each Hecke operator to its eigenvalue on \(f\). This map factors through the 
local component 
\(
  \frak{h}^\t{ord}_\mathfrak{m}
\)
of \(\mathfrak{h}^\t{ord}\), where \(\mathfrak{m}\) is the maximal ideal determined by \(f\). 
This local component
\[
  \frak{h}^\t{ord} \longrightarrow \frak{h}^\t{ord}_\frak{m}
\]
is the Hida family containing \(f\). Moreover, the induced specialization map 
\(\lambda_f\) on \(\frak{h}^\t{ord}_\frak{m}\) corresponds to an 
arithmetic map, which represents the specialization of the family 
at the parallel weight and nebentype of \(f\).\\

For explicit expressions and further details, check \cite{Nek06}, \cite{Hid06} or \cite{FJ24}.

\subsection{Big Galois representation}
To a Hida family \(\mathscr{G}\) one can attach a unique (up to equivalence) big Galois representation:
\[\rho_\mathscr{G} : G_F \rightarrow \t{GL}_2(\t{Frac}(\mcal{R})),\]
characterized by 
\[\det(1-\rho_\mathscr{G}(\t{Fr}_v) X) = 1 - \mathscr{G}(T_v) X + \mathscr{G}(\<{\varpi_v})|v|X^2,\]
for each prime \(v\nmid p\frak{n}\). See the full statement in \cite[Section 12.7]{Nek06}.\\

We denote by \(\bf{T}\) the associated rank 2 \(\mcal{R}\)-module with a continuous \(G_F\)-action. The key property of the big Galois representation is that, via specializations, one recovers the Galois representation attached to a Hilbert modular form in Section 2.4. This means that, applying an arithmetic map \(\nu_\frak{p}\) to the coefficients of \(\bf{T}\), there is an isomorphism
\[\bf{T}_\frak{p} := \bf{T}\otimes_{\mcal{R},\nu_\frak{p}}\overline{\bb{Q}}_p \simeq V(f_\frak{p}),\]
where \(V(f_\frak{p})\) is the Galois representation attached to the ordinary cuspform \(f_\frak{p}\) corresponding to the \(\frak{p}\)-specialization of the family.

\subsection{Critical character and twists}
The property
\[V(f)^\vee \simeq V(f)(1-k)\otimes \chi^{-1}\]
of the Galois representation \(V(f)\) attached to a parallel weight \(k\) Hilbert modular form \(f\) of character \(\chi\) does not successfully extend to big Galois representations, due to the explicit dependence of \(k\). For this reason, we consider a twisted big Galois representation \(\bf{T}^\dagger\) defined by
\[\bf{T}^\dagger := \bf{T}\otimes \Theta^{-1},\]
for certain character \(\Theta\) in such a way that, we get a cleaner relation
\[(\bf{T}^\dagger)^\vee \simeq \bf{T}^\dagger (-1),\]
normalized for (parallel) weight 2 Hilbert modular forms. This character \(\Theta\) is called \emph{critical twist}. See \cite{How06} or \cite{How07} for more details over \(\bb{Q}\) and an explicit construction. Here we study it for Hida families of Hilbert modular forms.\\

As in the previous subsection, consider the (branch of the) Hida family attached to \(f\) and the corresponding domain \(\mcal{R}\). Consider the homomorphism
\[\theta: \Gal(F[\infty]/F)=\t{Cl}_F^+(p^\infty) \longrightarrow \mcal{R}^\times,\]
determined by 
\[\theta^2(\sigma)=[\epsilon(\sigma)],\]
for each \(\sigma\in\Gal(F[\infty]/F)\). This induces a 1-dimensional big Galois representation
\begin{align*}
    \Theta: G_F &\longrightarrow \text{Gal}(F[\infty]/F) \overset{\theta}{\longrightarrow} \mcal{R}^\times.
\end{align*}
For this choice, \(\bf{T}^\dagger\) satisfies the desired property.

\subsection{Nekovar's extended Selmer group}
Using Selmer complexes, Nekovar defined generalized Selmer groups attached to Hida families. The construction we present in this article is a helpful tool for studying the arithmetic of these objects. In this subsection, we review the basic properties of these groups. We follow the exposition of \cite[Section 5.6]{LV10}. For precise definitions, see \cite{Nek06} or \cite[Section 5.1]{Fou12}.\\

Let \(M\in\{\bf{T},\ \bf{T}^\dag,\ V_\frak{p}^\dag\}\), all studied in this section, and let \(L/F\) be a number field. 
Write
\[\tilde{H}^1_f(L,M)\]
for the extended Selmer group in the sense of Nekovar.\\

Since \(M\) is \(p\)-ordinary, for each place \(v|p\), there is a short exact sequence  of \(G_{L_v}\)-modules
\[0\longrightarrow F_v^+(M) \longrightarrow M \longrightarrow F_v^{-}(M)\longrightarrow 0.\]

Recall Greenberg’s local conditions
\[H^1_\t{Gr} (L_v,M) := \begin{cases}
    \t{ker}\LL(H^1(L_v,M) \longrightarrow H^1(L_v^\t{unr},M)\RR) &\t{if } v\nmid p,\\
    \t{ker}\LL(H^1(L_v,M) \longrightarrow H^1(L_v,F_v^-(M))\RR) &\t{if } v|p,
\end{cases}\]
and the Greenberg Selmer group
\[ \t{Sel}_\t{Gr}(L,M) := \t{ker}\LL( H^1(L,M)  \longrightarrow \prod_v H^1(L_v,M)/H^1_\t{Gr} (L_v,M)\RR).\]

\begin{lemma}
    The group \(\tilde{H}_f^1(L,M)\) sits in the following short exact sequence
    \[ 0\longrightarrow \bigoplus_{v|p} H^0(L_v,F_v^-(M)) \longrightarrow \tilde{H}_f^1(L,M) \longrightarrow \t{Sel}_\t{Gr}(L,M) \longrightarrow 0.\]    
\end{lemma}

\begin{proof}
    This can be found in \cite{Nek06}. See, for example, its Section 0.10.
\end{proof}

In particular, if every place \(v\mid p\) is non-exceptional (in the sense that \(H^0(L_v,F_v^-(M))=0\), 
see \cite[Section 5.1.3]{Fou12}), then the correction term in the lemma vanishes and one has
\begin{corollary}
There is an isomorphism
    \[
\tilde{H}^1_f(L,M)\simeq\t{Sel}_\t{Gr}(L,M).
    \]
and, more precisely, for the representation \(V_\frak{p}^\dagger\), one has
\[
\tilde{H}^1_f(L,V_\frak{p}^\dagger)\simeq\t{Sel}_\t{Gr}(L,V_\frak{p}^\dagger)\simeq{H}^1_f(L,V_\frak{p}^\dagger).
\]
where \({H}^1_f(L,V_\frak{p})\) denotes Bloch--Kato's Selmer group.
\end{corollary}

\section{Heegner points via optimal embeddings}
In this section, we give the definition of Heegner point that we will be using. Our approach consistently relies on the theory of optimal embeddings, which is well-suited to both the definite (Gross points) and indefinite (Heegner points) cases.\\

As in the last section, \(F\) is a totally real field and \(p\) a fixed rational unramified prime. Fix also an archimedean prime \(\tau: F\hookrightarrow \bb{R}\). Let \(\frak{n}\) be an integral ideal in \(\mcal{O}_F\) and \(K/F\) be a CM-field of relative discriminant coprime with \(\frak{n}p\). Factorize \(\frak{n}=\frak{n}^+\frak{n}^-\) so that the primes dividing \(\frak{n}^+\) (resp. \(\frak{n}^-\)) are split (resp. inert) in \(K\). 

\begin{assumption}
    Assume \(\frak{n}^-\) to be square-free.
\end{assumption}

 Consider \(B\) the unique quaternion algebra of discriminant \(\frak{n}^-\) which is ramified at all the archimedean primes different from \(\tau\). If additionally \(B\) is ramified at \(\tau\), we say that we are in the \textbf{(totally) definite setting}. Otherwise, we are in the \textbf{indefinite setting}.

\subsection{Double cosets of points} We define sets \(X_m^{(K)}\), for \(m\geq 1\), which later will represent points of the Gross or Shimura curves, depending on the setting in which we are working.\\

Choose a set of isomorphisms \(\{ i_v: B_v:=B\otimes_F F_v \overset{\sim}{\longrightarrow} M_2(F_v)\}_{v\nmid \frak{n}^-}\). With a slight abuse of notation, we identify the elements of \(B_v\) with their respective matrices. Consider the standard family of Eichler orders \(\{R_m\}_{m\geq0}\) with respect to the choice \(\{i_v\}_{v\nmid \frak{n}^-}\) in a way that each Eichler order \(R_m\) is of level \(\frak{n}^+p^m\). Note that \(R_{m+1} \subset R_m\) for all \(m\geq 1\).\\

We are interested in two families of open-compacts encoding the level. Denote, for each \(m\geq 1\),
\[U_0(\frak{n}^+p^m) := \hat{R}_m^\times.\]
On the other hand, consider the subgroup of \(U_{0,1}(\frak{n}^+,p^m) \subset U_0(\frak{n}^+p^m)\) defined by
\[U_{0,1}(\frak{n}^+,p^m) := \LL\{x=(x_v)_v \in U_0(\frak{n}^+p^m) : x_p = 
\begin{pmatrix} 
    * & *\\ 
    0 & 1\\ 
\end{pmatrix}  \mod{p^m} 
\RR\}.\]

\begin{remark}
    These are the adelic version of the classical congruence subgroups \(\Gamma_0(N^+p^m)\) and \(\Gamma_{0,1}(N^+,p^m)\), extended here to the totally real quaternionic setting.
\end{remark}

Under the previous choices, we define, for every \(m\geq 1\) the double cosets
\begin{align*}
    X_m^{(K)} &:= B^\times \backslash (\Hom_{F,\tau}(K,B)\times \hat{B}^\times) \slash U_{0,1}(\frak{n}^+,p^m),
\end{align*}
where the left and right actions are given by
\[b \cdot [(f,g)] \cdot u := [(bfb^{-1},bgu)],\]
for \([(f,g)]\in X_m^{(K)}\), \(b\in B^\times\) and \(u\in U_{0,1}(\frak{n}^+,p^m)\).\\

Additionally, an action of the absolute Galois group \(G_K\) on \(X_m^{(K)}\) can be defined by
\[[(f,g)]^\sigma := [(f,\hat{f}(x_\sigma)g)],\]
where we see \(\sigma\) in \(\hat{K}\) via the Artin Reciprocity map
\[G_K \twoheadrightarrow \Gal(K^\t{ab}/K) \overset{\sim}{\longrightarrow} K^\times \backslash \hat{K}^\times\]
normalized for the choice of the geometric Frobenius. This action is posed to be compatible with Shimura's reciprocity law.

\begin{remark}
    These open compact subgroups are ``sufficiently small'' in the sense of, for example, \cite[Definition 2.1]{Fou12}.
\end{remark}

\begin{remark}
The order in the double coset is the inverse of that in \cite{LV10}, but follows the convention of \cite{Fou12} and \cite{CL16} instead.
\end{remark}

\subsection{Hecke operators on \(X_m^{(K)}\)}

We define an action of the Hecke operators on the double coset spaces \(X_m^{(K)}\), for each \(m\geq 1\).\\

During this subsection, let \(U\) denote \(U_{0,1}(\frak{n}^+,p^m)\) for any \(m\). Choose \(\gamma\in \hat{B}^\times\). Since 
\[U/(U\cap \gamma^{-1}U\gamma)\] 
is finite group, the double coset \(U\gamma U\) decomposes as a disjoint union
\[U \gamma U = \bigcup  \gamma_i U,\]
for a fixed finite choice of representatives \(\{\gamma_i\}_i\). Define the double coset operator
\begin{align*}
    [U\gamma U]: \t{Div}(X(U)) &\longrightarrow \t{Div}(X(U))\\
    [(f,g)]&\longmapsto \sum_{i} [(f,g\gamma_i)].
\end{align*}

Choose \(a\in \hat{F}^\times\), which embeds diagonally into \( \hat{B}^\times\). The corresponding Diamond operator \(\<{a}\) acts like 
\[\<{a}([(f,g)]):= [(f,ga)].\]

On the other hand, for a prime \(v\nmid p\), there is a decomposition 
\[
    U \begin{pmatrix} 
        \varpi_v & 0\\ 
        0 & 1\\ 
    \end{pmatrix} U =
    \bigcup_{a\in k_v} \begin{pmatrix} 
        \varpi_v & a\\ 
        0 & 1\\ 
    \end{pmatrix} U \bigcup \begin{pmatrix} 
        1 & 0\\ 
        0 & \varpi_v\\ 
    \end{pmatrix} U.
\]

while for \(\frak{p}|p\), there is a decomposition
\[
    U \begin{pmatrix} 
        \varpi_\frak{p} & 0\\ 
        0 & 1\\ 
    \end{pmatrix} U =
    \bigcup_{a\in k_\frak{p}} \begin{pmatrix} 
        \varpi_\frak{p} & a\\ 
        0 & 1\\ 
    \end{pmatrix} U.
\]

Denote, for \(v\nmid p\) prime,
\[\lambda_{v,a} = \begin{pmatrix} 
        \varpi_v & a\\ 
        0 & 1\\
    \end{pmatrix}\in \hat{B}^\times, \hspace{1cm} \lambda_{v,\infty} = \begin{pmatrix} 
        1 & 0\\ 
        0 & \varpi_v\\
    \end{pmatrix}\hat{B}^\times,\]
both defined to be \(1\) outside \(v\). The Hecke operators act, for primes \(v\nmid \frak{n}p\), as
    \[T_v([(f,g)]) := \sum_{a\in k_v} [(f,g\lambda_{v,a})] + [(f,g \lambda_{v,\infty})],
    \]
    and for \(\frak{p}|p\),
    \[U_\frak{p}([(f,g)]) := \sum_{a\in k_\frak{p}} [(f,g\lambda_{\frak{p},a})].
    \]
    
Define also the operator corresponding to \(p\)
\[U_p=\prod_{\frak{p}|p} U_\frak{p}.\]
In this case, one can also find a decomposition
\[
    U_m \lambda_{p,0} U_m = \bigcup_{a\in \mcal{O}_F/p\mcal{O}_F} U_m \lambda_{\frak{p},a},
\]
where, similarly as before, \(\lambda_{p,a}\) is the element with \(p\)-component \(\begin{pmatrix} 
    p & a\\ 
    0 & 1\\ 
\end{pmatrix}\) and value 1 everywhere else.

\subsection{Definition of Heegner points}
Now we introduce what it means to be a Heegner point for the double cosets \(X_m^{(K)}\). For any ideal \(\frak{c}\) of \(\mcal{O}_F\), relative to the CM extension \(K/F\), the order of conductor \(\frak{c}\) in \(K\) defined by
\[ \mcal{O}_\frak{c} := \mcal{O}_F + \frak{c}\mcal{O}_K.\]

\begin{definition}
    For an order \(\mcal{O}\) in \(K\) and an Eichler order \(R\) in \(B\), we say that a \(F\)-algebra embedding \(f\in\Hom_F(K,B)\) is an \textbf{optimal embedding} of \(\mcal{O}\) into \(R\) if it satisfies
    \[f(\mcal{O}) = R \cap f(K).\]
\end{definition}

To simplify the notation, denote \(U_m:=U_{0,1}(\frak{n}^+,p^m)\) and as usual \(U_{m,v} := U_m\otimes_{\mcal{O}_F} \mcal{O}_{F,v}\) for any finite prime \(v\) in \(F\) or \(U_{m,p}:=U_m \otimes_{\bb{Z}} \bb{Z}_p\) for the fixed rational prime \(p\).

\begin{definition}[Heegner point]\
    We say that \(P=[(f,g)]\in X_m^{(K)}\) is a \textbf{Heegner point of conductor \(\frak{c}\)} when \(f\) is an optimal embedding of \(\mcal{O}_\frak{c}\) into the order \(B\cap g\hat{R}_m g^{-1}\), and, for the \(p\)-component, it is satisfied 
    \begin{equation}
        f_p(\mcal{O}_{\frak{c},p}^\times \cap (1+p^m \mcal{O}_{K,p})^\times) = g_pU_{m,p}g_p^{-1} \cap f_p(\mcal{O}_{\frak{c},p}^\times).
    \end{equation}
\end{definition}

\subsection{Fields of rationality}
In this subsection, we study the fields \(\tilde{H}_m\), for \(m\geq1\), which are closely connected to the Galois action on Heegner points. We then relate them to certain ring class field and ray class field extensions.\\

Denote \(F[m]\) the ray class field of \(F\) with modulus \(p^m \infty\) (where \(\infty\) denotes the product of all the archimedean places). This is the extension of \(F\) given via class field theory by the narrow class group modulus \(p^m\)
\[\t{Gal}(F[m]/F) = \t{Cl}_F^+(p^m) = F^\times \backslash \bb{A}_F^\times / (F\otimes_\bb{Q} \bb{R})^\times_+ (1+p^m\hat{\mcal{O}}_F).\]

In general, for any field extension \(L/F\), we write \(L[m]\) for the compositum \(L\cdot F[m]\). This should not be confused with the respective ray class field extension of \(L\).\\

For our CM-extension \(K/F\) and an ideal \(\frak{c}\subset \mcal{O}_F\), define via class field theory the ring class field \(H_\frak{c}\) of conductor \(\frak{c}\) as the field extension with Galois group
\[\Gal(H_\frak{c}/K) = \widehat{K}^\times/K^\times \widehat{\mcal{O}}_\frak{c}^\times.\]

From now on fix \(\frak{c}\), not necessarily coprime with \(\frak{n}p\) and such that all primes \(v|\frak{c}\) are inert in \(K\). Throughout this paper, and whenever no ambiguity arises, we denote the field \(H_{\frak{c}p^m}\) by \(H_m\) and the order \(\mcal{O}_{\frak{c}p^m}\) by \(\mcal{O}_m\).\\

For \(m\geq 0\), define the group
\begin{equation}\label{identidad Z_m}
    Z_m := \hat{f}^{-1} (\hat{f} (\hat{\mcal{O}}^\times_{m}) \cap g U_m g^{-1})\subset \hat{K}^\times.
\end{equation}
It follows from the definition that 
\[ Z_m = \{ x=(x_v)_v\in \hat{\mcal{O}}_m^\times : x_p \equiv 1 \mod{p^m \mcal{O}_{K,p}}\}.\]

\

Observe \(Z_m\) has finite index in \(\hat{\mcal{O}}_{m}\) (and therefore is open), so it is the finite part of a norm group. By class field theory, define the field \(\tilde{H}_m\) determined by 
\begin{align*}
  \hat{K}^\times / K^\times Z_m = \Gal(\tilde{H}_m/K).
\end{align*}

\begin{proposition}\label{Prop. Inclusion}
    There is a field inclusion 
\[\tilde{H}_m \subset H_{m}[m].\]
\end{proposition}

\begin{proof}
Denote 
\begin{align*}
    V_m :&= \text{finite part of } \t{Norm}_{K[m]/K}(\bb{I}_{K[m]}),\\
    W_m :&= \text{finite part of } \t{Norm}_{F[m]/F}(\bb{I}_{F[m]}).
\end{align*}
Comparing the two Galois groups 
\begin{align*}
    \Gal(\tilde{H}_{m}/K)&=\hat{K}^\times/K^\times Z_m,\\
    \Gal(H_{m}[m]/K)&= \hat{K}^\times/K^\times (V_m \cap \hat{\mcal{O}}_{m}^\times),
\end{align*}
in order to have \(\tilde{H}_m \subset H_m[m]\) it is enough showing 
\[K^\times(V_m \cap \hat{\mcal{O}}_{m}^\times) \subset K^\times Z_m.\] 
Take any \(x\in V_m\cap \hat{\mcal{O}}_{m}^\times\). What we have to do is show
\[x_p \equiv 1 \mod {p^m\mcal{O}_F}.\] 
Since \(x\in \hat{\mcal{O}}_{m}^\times\), we can express it as 
\[x=\alpha+cp^m\beta\] 
for some \(\alpha\in \hat{\mcal{O}}_F\), \(\beta\in \hat{\mcal{O}}_K\) and \(c\in\frak{c}\). In particular, since 
\[\t{Norm}_{K/F}(x_p)\equiv \alpha_p^2 \mod {p^m\mcal{O}_F},\] 
is enough showing that 
\[\alpha_p^2\equiv 1 \mod {p^m\mcal{O}_F}.\]
Let \(W_{m,\infty} := W_m \times \bb{R}_+^d\) and observe  
\begin{align*}
    \Gal(K[m]/F)&= \bb{I}_F / F^\times (W_{m,\infty}\cap \t{Norm}_{K/F}(\bb{I}_K)) \\&= \bb{I}_F / F^\times \t{Norm}_{K[m]/F}(\bb{I}_{K[m]}).
\end{align*}
Focusing on the finite part, we have
\[F^\times \t{Norm}_{K/F}(V_m) \subset F^\times W_m.\] 
Observe
\[\t{Norm}_{K/F}(x)\in F^\times W_m \cap \hat{\mcal{O}}_F =W_m,\] so it must have 
\[\t{Norm}_{K/F}(x_p)=1 \mod {p^m\mcal{O}_F}.\]
\end{proof}

\begin{remark}
    This result is analogous to \cite[Proposition 3.3]{LV10}. However, over \(\bb{Q}\), they are able to prove equality between the two fields, while the author believes that this no longer holds in general for totally real fields.
\end{remark}

\subsection{Galois action on Heegner points} \label{Sección propiedades}
In this section we prove the following relations regarding the action of the Hecke operators and Galois action on Heegner points (Proposition \ref{Propiedad 1} and Proposition \ref{Propiedad 2} below):
\begin{align*}
    U_p (P) &= \t{tr}_{H_{{m+1}}[m+1]/H_{m}[m+1]} (Q),\\
    T_v (P) &= \t{tr}_{H_{\frak{c}vp^m}[m]/H_{\frak{c}p^m}[m]} (Q).
\end{align*} 
Here \(P\in X_m^{(K)}\) is a Heegner point and \(Q\) is a point in the support of \(U_p\) (resp. \(T_v\)).\\

We will need a couple of lemmas.

\begin{lemma}[Dedekind formula] \label{Dedekind formula}
    \[|\t{Pic}(\mcal{O}_\frak{c})| = \frac{h_K |\frak{c}| \prod_{v|\frak{c}}\LL(1-\LL(\frac{K}{v}\RR)|v|^{-1}\RR)}{[\mcal{O}_K^\times:\mcal{O}_\frak{c}^\times]}.\]
\end{lemma}

\begin{proof}
    See \cite[Section 2]{Lon12}. 
\end{proof}

\begin{lemma} \label{Linearly disjoint}
    For \(n\geq1\), \(H_{{n+1}}\) and \(H_{n}[n+1]\) are linearly disjoint over \(H_{n}\).
\end{lemma}

\begin{proof}
    The result follows from comparing the Galois group
    \[
    \Gal(H_{n+1}/H_n)\simeq \frac{(\mcal{O}_{K}/p^{n+1}\mcal{O}_K)^\times}{(\mcal{O}_F/p^{n+1}\mcal{O}_F)(\mcal{O}_{K}/p^{n}\mcal{O}_K)^\times},\]
    with
    \[\Gal(H_{n}[n+1]/H_n)\simeq \frac{(\mcal{O}_{F}/p^{n+1}\mcal{O}_F)^\times}{(\mcal{O}_F/p^{n}\mcal{O}_F)}\hookrightarrow\frac{(\mcal{O}_{K}/p^{n+1}\mcal{O}_K)^\times}{(\mcal{O}_K/p^{n}\mcal{O}_K)}.
    \]
\end{proof}

\begin{proposition} \label{Propiedad 1}
    For \(P\in X_m^{(K)}\) a Heegner point of conductor \(\frak{c}p^n\) and \(Q\) a summand in \(U_p(P)\). We have
    \[U_p (P) = \t{tr}_{H_{{n+1}}[n+1]/H_{n}[n+1]} (Q).\]
\end{proposition}

\begin{proof} 
    Denote \(P=[(f,g)]\). The element \(Q\) can be expressed by hypothesis as \(Q=[(f,g\lambda_{p,a})]\), where 

    \[ \lambda_{p,a} = \begin{cases}
    \begin{pmatrix} p & a \\ 0 & 1\\ \end{pmatrix} & \text{at } p,\\
    1  & \text{at } v\neq p,
    \end{cases}\]
    for some \(a\in \mcal{O}_F/p\mcal{O}_F\). Take any
    \begin{align*}
    \sigma\in \Gal(H_{n+1}[n+1]/H_n[n+1]) &= K^\times (V_{n+1}\cap \hat{\mcal{O}}^\times_{n})  / K^\times (V_{n+1}\cap \hat{\mcal{O}}^\times_{n+1}).    
    \end{align*}
    
    Using similar arguments as in the proof of Proposition \ref{Prop. Inclusion}, we show that \(\sigma\) can be represented by some \(x\in Z_n\). Using that \([(f,g)]\) is a Heegner point of conductor \(\frak{c}p^n\), we know using (\ref{identidad Z_m}) that \(\hat{f}(x) = g u g^{-1}\) for some \(u\in U_n\), so we can express 
    \begin{align*}
        Q^\sigma &= [(f, \hat{f}(x) g \lambda_{p,a} )] \\ 
        &= [(f,gu\lambda_{p,a} )]\\ 
        &= [(f,gu\lambda_{p,a}u^{-1} )]\\ 
        &= [(f, g\lambda_{p,b})] \in U_p(P),
    \end{align*}
    for some \(b\) in \(\mcal{O}_F/p\mcal{O}_F\). All we have to check is that, in this way, we get all the \(|\mcal{O}_F/p\mcal{O}_F|\) summands. From Lemma \ref{Linearly disjoint} follows
    \[\Gal(H_{n+1}[n+1]/H_n[n+1]) \simeq \Gal(H_{n+1}/H_n),\]
    and this last group has exactly \(|\mcal{O}_F/p\mcal{O}_F|\) elements due to Dedekind's formula (Lemma \ref{Dedekind formula}).
\end{proof}

\begin{proposition} \label{Propiedad 2}
    Let \(v\nmid \frak{nc}p^m\) be a prime that is inert in \(K\). Then
    \[T_v (P) = \t{tr}_{H_{\frak{c}vp^n}[n]/H_{\frak{c}p^n}[n]} (Q).\]
\end{proposition}

\begin{proof}
    The proof follows just as in Proposition \ref{Propiedad 1}. Note that since \(v\) is inert in \(K\), Dedekind's formula predicts extra summands, which fits the behaviour of the operator \(T_v\).
\end{proof}

\section{Gross curves and Shimura curves over totally real  fields} \label{Gross curves and Shimura curves over totally real  fields}

The double cosets \(X^{(K)}_m\) correspond to sets of points of different curves depending on whether we are in the definite or indefinite settings. We denote the curve attached to \(X_m^{(K)}\) as \(X_m\). During this subsection, we explain this fact and give a description of both curves. See \cite{LV10} or \cite{BD96} for details over \(\bb{Q}\).\\

As in previous sections, fix a quaternion algebra \(B/F\) and an embedding \(\tau:F\hookrightarrow \bb{R}\) such that \(B\otimes_{F,\sigma} \bb{R}= \bb{H}\) for all \(\sigma\neq\tau\). Denote \(B_\infty:= B\otimes_{F,\tau} \bb{R}\). Recall also the congruence subgroup \(U_{0,1}(\frak{n}^+,p^m)\).\\

In both settings, we have the following natural diagram between curves, which is now presented in order to fix the name of the maps.

\[\begin{tikzcd}
	\dots && {X_m} && {X_{m-1}} && \dots
	\arrow["{\alpha_{m+1}}", from=1-1, to=1-3]
	\arrow["{\alpha_{m}}", from=1-3, to=1-5]
	\arrow["{\alpha_{m-1}}", from=1-5, to=1-7]
\end{tikzcd}\]

\subsection{Totally definite setting: Gross curves}

In this setting \(B_\infty = \bb{H}\). Consider the homogeneous space defined by
\[Y(A):=\{x\in B\otimes_{F,\tau} A | x\neq 0, \t{Norm}(x)=\t{Trace}(x)=0\}/A^\times,\]
where \(A\) is any \(F\)-algebra. The group \(B^\times\) acts on \(Y\) via conjugation. This object substitutes the upper- and bottom-half planes in the definite setting.

\begin{lemma}\label{Identificacion 1}
    There is a \(B^\times\)-equivariant identification between \(\Hom_{F,\tau}(K,B)\) and \(Y(K)\).
\end{lemma}

\begin{proof}
    The proof is the same as in the rational case, since we are only considering the factor coming from \(\tau\). See \cite[Section 2.1]{LV10}, \cite[Section 1]{BD96} or \cite[Section 3]{Gro87}.
\end{proof}

\begin{definition}[Gross curve]
The \(\textbf{Gross curve}\) (also called \textbf{definite Shimura curve}) of level \(U_{0,1}(\frak{n}^+,p^m)\) is defined as the double coset
    \begin{align*}
    X_m := B^\times \backslash (Y \times \hat{B}^\times)/ U_{0,1}(\frak{n}^+,p^m),
\end{align*}
where the action of \(\hat{B}^\times\) on \(Y\) is trivial.
\end{definition}

Notice the similarities with classical Shimura curves. Also, Lemma \ref{Identificacion 1} brings us the following corollary that gives us the relation between the sets \(X_m^{(K)}\) studied along the article and Gross curves.
\begin{corollary}
    \(X_m^{(K)}=X_m(K)\). 
\end{corollary}

The geometry of Gross curves is less intricate than that of Shimura curves.  
In fact, the curves \(X_m\) admit a simpler description as a disjoint union of conics.   

\begin{lemma}
Let \(m\geq1\), set \(U:=U_{0,1}(\frak{n}^+,p^m)\), and choose double-coset representatives
\[
\hat{B}^\times=\bigsqcup_{i=1}^{h(m)} B^\times g_i U,\qquad
\Gamma_m^i:=g_i U g_i^{-1}\cap B^\times.
\]
Then
\[
X_m = B^\times\backslash\bigl(Y\times \widehat{B}^\times\bigr)/U
= \bigsqcup_{i=1}^{h(m)} \Gamma_m^i\backslash Y.
\]
\end{lemma}

\begin{proof}
Using the disjoint union for \(\widehat{B}^\times\),
\[
X_m \simeq \bigsqcup_{i=1}^{h(m)} B^\times\backslash(Y\times g_i U)/U.
\]
Since \(U\) acts trivially on \(Y\), the map \([(y,g_i u)]\mapsto (y,g_i)\) identifies
\((Y\times g_i U)/U \simeq Y\).
Under this identification, the left \(B^\times\)-action on \(Y\times\{g_i\}\) has stabilizer
\(\{b\in B^\times : b g_i\in g_i U\}=\Gamma_m^i\), hence
\[
B^\times\backslash(Y\times g_i U)/U \simeq \Gamma_m^i\backslash Y.
\]
Taking the disjoint union over \(i\) gives the claim.
\end{proof}

One can define an action of the Hecke operators via correspondences described by Brandt matrices. 
Denote the resulting Hecke algebra by 
\[\bb{B}_m.\] 
This construction is explained in \cite[Section 1]{Gro87} over \(\bb{Q}\) and for quaternion algebras over general number fields in \cite[Chapter 41]{Voi21}.

\subsection{Indefinite setting: Shimura curves}
In this setting \(B\otimes_{F,\tau} \bb{R}=\t{M}_2(\bb{R})\).

The following description of the upper and lower half-planes plays a key role in viewing CM points as optimal embeddings. It also provides a uniform framework to treat simultaneously the (totally) definite and indefinite settings. This result is the analogue of Lemma \ref{Identificacion 1}, although it applies to the \(\bb{C}\)-points rather than the \(K\)-points.

\begin{lemma}\label{Identificacion 2}
    There is a \(B^\times\)-equivariant identification between \(\Hom_\bb{R}(\bb{C},B_\infty)\) and \(\bb{P}=\bb{C}\backslash \bb{R}\).
\end{lemma}

\begin{proof}
    As it happened in the definite setting, the proof is identical to the rational case. See \cite[Section 2.2]{LV10} or \cite[Section 1]{BD96}.
\end{proof}

The group \(B^\times\) acts on \(\mcal{H}\) via fractional linear transformations under the composition
\[B^\times \otimes_{F,\tau} \bb{R}= \t{GL}_2(\bb{R}) \curvearrowright  \mcal{H}.\]

\begin{definition}
The \textbf{(indefinite) Shimura curve} of level \(U_{0,1}(\frak{n}^+,p^m)\) is the complex variety defined by the double coset
    \[X_m(\bb{C})=B^\times \backslash (\mcal{H} \times \hat{B}^\times) / U_{0,1}(\frak{n}^+,p^m) \cup \{ \t{cusps} \}.\] 
\end{definition}

\begin{remark}
   There will only be cusps in the classical case: \(F=\bb{Q}\) and \(B=\t{M}_2(\bb{Q})\).
\end{remark}

The curve \(X_m\) is a, possibly disconnected, compact Riemann surface. It admits a canonical model defined over \(F\). The relation of \(X_m\) to the double coset \(X_m^{(K)}\) is a consequence of Lemma \ref{Identificacion 2}.

\begin{corollary}
    There is an injection \(X_m^{(K)} \hookrightarrow X_m(\bb{C})\).
\end{corollary}

\begin{proof}
    Embeddings \(K\hookrightarrow B\) are lifted to embeddings \(\bb{C} \hookrightarrow B_\infty\) by extending scalars to \(\bb{R}\) via \(\tau\).
\end{proof}

There is a Hecke action via correspondences on the points of the curve. Denote this Hecke algebra, natural to the quaternionic setting, by 
\[\bb{B}_m,\] same as in the definite setting. See \cite[Section 2.1.3]{Fou12} for more details on how this action is defined.

\section{Jacquet--Langlands correspondence}

The Jacquet--Langlands correspondence is a deep result in the theory of automorphic forms and representations. It is a key tool in this construction: It builds the bridge between classical Hilbert modular forms (considered in Section \ref{Review of Hilbert modular forms and Hida theory}) and Gross curves or Shimura curves (considered in Section \ref{Gross curves and Shimura curves over totally real  fields}). While, for the first one, we allowed some primes dividing the level \(\frak{n}\) to be inert in \(K\) (weak Heegner hypothesis), in the second, the level \(\frak{n}^+\) always satisfies the Heegner hypothesis and we are able to build Heegner points.

\subsection{Picard group and Jacobian}
\label{Picard group and Jacobian}
This short subsection introduces the notation that will be common to both settings and used later in the applications.\\

For \( X_m \) denoting either a Gross curve or a Shimura curve, let  \(\t{Pic}(X_m)\) (respectively, \(\t{Pic}^0(X_m)\)) be the Picard group (respectively, the degree-zero Picard group) of \( X_m \), that is, the quotient of the group of divisors (respectively, divisors of degree \(0\)) by the subgroup of principal divisors. \\

In what follows, define
\[
  J_m := \t{Pic}(X_m) \otimes_{\bb{Z}} \mcal{O},
  \qquad
  J_m^0 := \t{Pic}^0(X_m) \otimes_{\bb{Z}} \mcal{O},
\]
where \(\mcal{O}\) denotes the ring of integers of the Hecke field of \( f \). \\

The degree map induces a short exact sequence
\[
  0 \longrightarrow J_m^0 
  \longrightarrow J_m 
  \overset{\t{deg}}{\longrightarrow} \mcal{O} 
  \longrightarrow 0.
\]

Let
\[
  J_m^{\t{ord}} := e_m^{\t{ord}} \cdot J_m,
  \qquad
  J_m^{0, \t{ord}} := e_m^{\t{ord}} \cdot J_m^0
\]
denote the ordinary parts, and define the inverse limits
\[
  J_\infty^{\t{ord}} := \varprojlim_m J_m^{\t{ord}},
  \qquad
  J_\infty^{0, \t{ord}} := \varprojlim_m J_m^{0, \t{ord}}.
\]

For a Hida family \(\mcal{R}\) and its twist \(\mcal{R}^\dagger\), denote also
\[
  \mathbf{J}_m := J_m^{\t{ord}} \otimes_{\bb{T}_\infty^{\t{ord}}} \mcal{R},
  \qquad
  \mathbf{J}_m^\dagger := \mathbf{J}_m \otimes_{\mcal{R}} \mcal{R}^\dagger,l
\]
and
\[
  \mathbf{J} := \varprojlim_m \bf{J}_m,
  \qquad
  \mathbf{J}^\dagger := \mathbf{J} \otimes_{\mcal{R}} \mcal{R}^\dagger.
\]

\subsection{The \(\frak{n}^-\)-new part}

Recall from Section \ref{Review of Hilbert modular forms and Hida theory} the \(\bb{C}\)-vector space of Hilbert cuspforms of parallel weight \(k\) and level \(\frak{n}\) denoted by \(S_k(\frak{n};\bb{C})\), and its corresponding Hecke algebra, denoted by \(\frak{h}_k(\frak{n};\bb{C})\). Define the \(\frak{n}^-\)-new quotient
\[S_k^{\frak{n}^-\t{-new}}(\frak{n};\bb{C})\]
to be the subspace of \(S_k(\frak{n};\bb{C})\) made of cuspforms that are new at all the primes dividing \(\frak{n}^-\), and denote
\[\frak{h}_k^{\frak{\frak{n}^-}\t{-new}}(\frak{n};\bb{C})\]
its corresponding Hecke algebra. 
In particular, for the spaces \(S_k(\frak{n}p^m)\) considered in Section \ref{Review of Hilbert modular forms and Hida theory}, we are interested in its \(\frak{n}^-\)-new Hecke algebra, denoted by \(\bb{T}_m\). When studying Hida theory, one considers the big ordinary version
\[\bb{T}_m^\t{ord} := e^\t{ord}\cdot \bb{T}_m, \hspace{1cm} \bb{T}^\t{ord}_\infty:=\varprojlim \bb{T}_m^\t{ord}.\]

From the properties of \(\frak{h}_\infty^\t{ord}\), one deduces the \(\Lambda\)-structure of \(\bb{T}_\infty^\t{ord}\) and the independence of the weight. There is a canonical projection
\[ \frak{h}_\infty^\t{ord} \longrightarrow \bb{T}_\infty^\t{ord}.\]

\subsection{Correspondence between Hecke algebras}
In this subsection, we apply the Jacquet--Langlands to our setting.\\

Recall from the last section the Hecke algebras \(\bb{B}_m\) which act on Shimura/Gross curves, denoted in the same way for both the definite and indefinite settings. The Jacquet--Langlands provides us with the following relation.
\begin{theorem}[Jacquet--Langlands correspondence]
    For every \(m\), there is an isomorphism of \(\mcal{O}\)-algebras
    \[\t{JL}_m: \bb{T}_m \overset{\sim}{\longrightarrow} \bb{B}_m.\]  
\end{theorem}

\begin{proof}
    See \cite[Section 2.3.6]{Hid06}.
\end{proof}

As it is done in \cite{LV10} in the rational case, we need to extend this result to fit with Hida theory. First, define a \(\Lambda\)-algebra structure on \(\bb{B}_m\) by letting the group element \([z] 
\) act by the Diamond operator \(\<{z}\). One can now have a notion of \(\bb{B}^\t{ord}\). Define the inverse limit \(\bb{B}^\t{ord}_\infty\) over all \(m\). 

\begin{proposition}
    The Jacquet-Langlands correspondence extends naturally to an isomorphism of \(\Lambda\)-algebras
\[\t{JL}_\infty^\t{ord}: \bb{T}_\infty^\t{ord} \overset{\sim}{\longrightarrow} \bb{B}_\infty^\t{ord}.\]
\end{proposition}
\begin{proof}
    See \cite[Proposition 6.1]{LV10}. 
\end{proof}

\subsection{Correspondence between Galois representations}
We restrict ourselves to the \mbox{indefinite} setting for the rest of the section. The goal is to compare, using Jacquet--Langlands, the Galois representation arising from towers of Hilbert modular varieties with the one arising from towers of Shimura curves.\\

Recall from Section \ref{Review of Hilbert modular forms and Hida theory} the big Galois representations \(\bf{T}\) and \(\bf{T}^\dagger\). Let us review its construction. Consider the Hilbert modular variety of level \(\mcal{U}_{0,1} (\frak{n},p^m)\) as in \cite[Section 4]{FJ24}
\[Y_m:= Y_{0,1}(\frak{n},p^m):= \t{GL}_2(F)\backslash (\mcal{H}^d\times \t{GL}_2(\hat{F}))/\mcal{U}_{0,1}(\frak{n},p^m).\]

Then \(\bf{T}\) is obtained from inverse limits of ordinary parts of middle-degree étale cohomology in the following way
\begin{align*}
    \t{Ta}_{p,m}^\t{ord} &:= e^\t{ord} \cdot H^d_\t{ét}(Y_{0,1}(\frak{n},p^m)_{\overline{F}} , \bb{Z}_p) \otimes_{\bb{Z}_p} \mcal{O},\\
    \t{Ta}_p^\t{ord} &:= \varprojlim_m \t{Ta}_{p,m}^\t{ord},\\
    \bf{T}&:= \t{Ta}_p^\t{ord}\otimes_{\frak{h}_\infty^\t{ord}}\mcal{R}.
\end{align*}
Recall the twisted version \(\bf{T}^\dagger:=\bf{T}\otimes_\mcal{R}\mcal{R}^\dagger\). In Section \ref{Review of Hilbert modular forms and Hida theory} we explained the properties of these representations. We want to relate it with the representation
\begin{align*}
    \underline{\t{Ta}}_{p,m}^\t{ord}&:= e^\t{ord}\cdot \t{Ta}_p(\t{Jac}(X_{0,1}(\frak{n},p^m)))\otimes_{\bb{Z}_p}\mcal{O},\\
    \underline{\t{Ta}}_p^\t{ord} &:= \varprojlim_m \underline{\t{Ta}}_{p,m}^\t{ord},\\ 
    \underline{\bf{T}}&:= \underline{\t{Ta}}_p^\t{ord}\otimes_{\bb{T}_\infty^\t{ord}}\mcal{R},\\
    \underline{\bf{T}}^\dagger&:=\underline{\bf{T}}\otimes_\mcal{R}\mcal{R}^\dagger.
\end{align*}
Note the difference, not only in the curve, but also in the use of the Hecke algebra \(\bb{T}^\t{ord}_\infty\) instead of \(\frak{h}_\infty^\t{ord}\).

\begin{proposition}\label{Comparacion representaciones}
    There are canonical \(G_F\)-equivariant \(\mcal{R}\)-linear isomorphisms
    \[\bf{T} \simeq \underline{\bf{T}}\hspace{5mm} \t{and} \hspace{5mm} \bf{T}^\dagger \simeq \underline{\bf{T}}^\dagger.\]
\end{proposition}
Note that this also applies to the corresponding specializations.

\begin{proof}
    See \cite[Proposition 6.4]{LV10} for a similar discussion.  The idea is the following: Jacquet--Langlands identifies the ordinary Hecke algebras on both sides. For each arithmetic specialization, the Hilbert and quaternionic Galois representations have the same Frobenius characteristic polynomials at almost all places, hence are isomorphic. Since this holds on a Zariski-dense set and the residual representation is irreducible, their big representations coincide. 
\end{proof}

\section{Construction of big Heegner points}
\label{Construction of big Heegner points}

In this section, we demonstrate that Heegner points exist in the sense described above and that it is possible to construct a family meeting the necessary conditions for an Euler system. Later we build big Heegner points attached to Hida family of Hilbert modular forms.

\subsection{A compatible family of Heegner points}\label{Capítulo construcción}

We first fix a model of \(B/K\) in order to have a canonical injection \(\iota_K:K\hookrightarrow B\). See \cite{Wan22} or \cite{Hun17} for more details in the totally real setting.\\

 Fix a CM type, that is, a subset \(\Sigma\subset \t{Hom}(K,\bb{C})\) such that \(\Sigma \cup \overline{\Sigma} = \t{Hom}(K,\bb{C})\) and \(\Sigma \cap \overline{\Sigma} =\emptyset\). As in the last section, choose \(\theta\in K\) such that
\begin{enumerate}
    \item \(\t{Im}(\sigma(\theta))>0\), \(\forall \sigma \in \Sigma\).
    \item \(\{1,\theta_v\}\) is a basis of \(\mcal{O}_{K.v}\) over \(\mcal{O}_{F,v}\) for all primes \(v|d_{K/F}p\frak{n}\).
    \item \(\theta\) is a local uniformizer at primes \(v\) ramified in \(K\).
\end{enumerate}

We fix a basis \(B=K \oplus K  j\) satisfying
\begin{enumerate}
    \item \(j^2=\beta\in F^\times\) with \(\sigma(\beta)<0\) for all \(\sigma\in \Sigma\) and such that for all \(x\in K\) holds \(jx=\overline{x}j\).
    \item \(\beta \in (\mcal{O}^\times_{F,v})^2\) for all \(v|p\frak{n}^+\) and \(\beta\in\mcal{O}_{F,v}^\times\) for all \(v|d_{K/F}\).\\
\end{enumerate}

For \(v|p\frak{n}^+\), fix the isomorphisms \(i_v: B_v\overset{\sim}{\rightarrow} M_2(F_v)\) determined by
\[i_v(\theta) =
\begin{pmatrix} 
    T(\theta) & -N(\theta) \\ 
    1 & 0\\
\end{pmatrix}, \hspace{1cm}
i_v(j) = \sqrt{\beta}
\begin{pmatrix} 
    -1 & T(\theta) \\ 
    0 & 1\\
\end{pmatrix}.
\]

Fix \(\frak{c}\) coprime with \(\frak{n}^+\) and such that all \(v|\frak{c}\) are inert in \(K\). Let \(\frak{n}^+=\frak{N}^+\overline{\frak{N}^+}\) for the choice of CM type. For each \(m\geq 0\), consider the elements \(b_\frak{c}^{(m)} \in \hat{B}^\times\) whose local components \(b_\frak{c}^{(m)}=\prod_v b_v^{(m)}\) are defined by
\begin{itemize}
    \item If \(v\nmid p \frak{c}\frak{n}^+\), then:
    \[b_v^{(m)} := 1.\]
    \item If \(v|\frak{n}^+\), then:
    \[ b_v^{(m)} := \frac{1}{\theta-\overline{\theta}}\begin{pmatrix} 
        \theta & \overline{\theta} \\ 
        1 & 1\\
    \end{pmatrix} \in \t{GL}_2(K_w) = \t{GL}_2(F_v),\ v=w \overline{w}.\]

    \item If \(v|\frak{c}\), write \(\frak{c}=(\varpi_\frak{c})^{c_v}\) for \(\varpi_\frak{c}\) an uniformizer of \(\frak{c}\). Write also \(r_v=c_v-\t{ord}_v\frak{n}\). Then
    \[ b_v^{(m)} :=\begin{pmatrix} 
        \varpi_v^{-r_v} & 0 \\ 
        0 & 1\\
    \end{pmatrix}.\]

    \item If \(v|p\) and \(v\) splits in \(K\), then:
    \[ b_v^{(m)} := \begin{pmatrix} 
        \theta & -1 \\ 
        1 & 0\\ 
    \end{pmatrix}\begin{pmatrix} 
        \varpi_v^m & 0 \\ 
        0 & 1\\ 
    \end{pmatrix}\in\t{GL}_2(K_w) = \t{GL}_2(F_v),\ v=w \overline{w}.\]
    \item If \(v|p\) and \(v\) is inert in \(K\), then:
    \[ b_v^{(m)} := \begin{pmatrix} 
        0 & 1 \\ 
        -1 & 0\\ 
    \end{pmatrix}\begin{pmatrix} 
        \varpi_v^m & 0 \\ 
        0 & 1\\ 
    \end{pmatrix}.\]
\end{itemize}

\begin{proposition}
    \({P}_{\frak{c}p^n,m} := [(\iota_K, b_\frak{c}^{(m)})] \in X^{(K)}_m\) is a Heegner point of conductor \(\frak{c}p^{n+m}\).
\end{proposition}

\begin{proof}
    Direct computation. This is a combination of \cite[Theorem 1.6]{CL16} and \cite[Section 3]{Hun17}.
\end{proof}

\begin{remark}
    One can recover this explicit construction over \(\bb{Q}\) from that originally proposed in \cite[Section 4]{LV10} by choosing \(g=\iota_K\), which coincides locally at \(p\) with their \(\psi_p^{(c)}\). In this case, taking \(1\) as the representative of the identity in \(\hat{R}_m\backslash \hat{B}^\times/B^\times\), the elements \(a^{(c,m)}\) coincide with our \(b^{(m)}\). Note that in \cite{LV10}, the choice of the open-compact subgroup differs from the one considered here or in \cite{Fou12} and \cite{Wan22}, since they define it to act on the left.
\end{remark}

\begin{proposition}\label{Compatibilidad}
    The following identities hold for our particular choices \(P_{\frak{c},m}\):
    \begin{enumerate}
        \item Horizontal compatibility:
        \begin{align*}
            U_p(P_{\frak{c}p^{r-1},m}) &= \text{tr}_{H_{m+r}[m+r]/H_{m+r-1}[m+r]}(P_{\frak{c}p^r,m}),\\
            T_v(P_{\frak{c},m}) &= \text{tr}_{H_{\frak{c}vp^m}[m]/H_{\frak{c}p^m}[m]}(P_{\frak{c}v,m}), \t{ for }v\nmid \frak{n}p \t{ inert prime}.
        \end{align*}
        \item Vertical compatibility:
        \[U_p(P_{\frak{c},m-1}) = \tilde{\alpha}_{m} \LL(\text{tr}_{H_{m}[m]/H_{{m-1}}[m]}(P_{\frak{c},m}) \RR).\]
   \end{enumerate}
\end{proposition}

\begin{proof}
    The proof is derived from Propositions \ref{Propiedad 1} and \ref{Propiedad 2}. For more details, see \cite[Proposition 4.7]{LV10} and \cite[Proposition 4.8]{LV10}. Similarly, in the indefinite setting, check \cite[Proposition 4.5]{Fou12}. 
\end{proof}

The last property provides an explicit description of the action of the Galois group \(\Gal(\overline{F},H_{m})\) on these points. This verification is necessary because our points are rational over an extension of \(H_{m}\), rather than over \(H_{m}\) itself, as is the case for the classical CM points studied through Shimura's reciprocity law.

\begin{proposition}
For any \(\sigma\in \Gal(\overline{F}/H_{m})\) and \(m\geq1\), there is an equality
\[
    P_{\frak{c},m}^\sigma = \<{\vartheta(\sigma)} P_{\frak{c},m},
\]
where \(\vartheta(\sigma)\) is a square root of
\[G_F^{ab}\longrightarrow \Gal(F[\infty]/F) \longrightarrow 1+p\mcal{O}_{F,p}.\]
\end{proposition}

\begin{proof}
    This proof adapts \cite[Corollary~2.2.2]{How06} and \cite[Section~4.4]{LV10}. It suffices to show that such \(\sigma\) can be represented (and therefore completely determined) by \(\vartheta(\sigma)\).

    Since the points \(P_{\frak{c},m}\) are rational over \(H_m[m]\), we may restrict to \(\sigma\in \Gal(H_m[m]/H_m)\). Take \(x\in \widehat{K}^\times\) a representative of \(\sigma\), and multiply by an element acting trivially on the level structure so that
    \[
        x_p = \alpha + c p^m \beta,
    \]
    with \(\alpha\in \mcal{O}_{F,p}^\times\), \(\beta\in \mcal{O}_{K,p}\) and \(c\in\frak{c}\), 1 elsewhere.

    The map 
    \[
        \Gal(H_m[m]/H_m) \hookrightarrow \Gal(F[m]/F[0])
    \]
    is given idelically by the norm, which yields
    \[
        \t{Norm}_{K/F}(x_p)\equiv \alpha^2 \pmod{p^m\mcal{O}_F}.
    \]
    There is a natural morphism
    \[
        \Gal(F[m]/F[0]) \longrightarrow (\mcal{O}_F/p^m\mcal{O}_F)^\times
    \]
    whose kernel acts trivially on the relevant points. Altogether, this implies that the action of \(\sigma\in \Gal(H_m[m]/H_m)\) on the points is represented by \(\alpha\) modulo units, and therefore
    \[
        P_{\frak{c},m}^\sigma 
        = [(\iota_K,b_\frak{c}^{(m)})]^\sigma 
        = [(\iota_K,\widehat{\iota}_K(x)b_\frak{c}^{(m)})] 
        = [(\iota_K,\widehat{\iota}_K(\alpha)b_\frak{c}^{(m)})] 
        = \<{\alpha}P_{\frak{c},m} 
        = \<{\vartheta(\sigma)}P_{\frak{c},m}.\]
\end{proof}

\subsection{Big Heegner points}
For the rest of the section, we vary the points in Hida families by defining an action of Hida's big ordinary Hecke algebra. We define big Heegner points and extend the compatibility relations we proved into a more familiar Euler system relations.\\

For \(\frak{c}\) an ideal in \(\mcal{O}_F\), denote
\(\mcal{G}_\frak{c} = \Gal(K^\t{ab}/H_\frak{c})\)
and \(\mcal{G}_0 = \Gal(K^\t{ab}/K)\). Furthermore, for \(\frak{c}\) fixed as in previous sections, denote \(\mcal{G}_m:=\mcal{G}_{\frak{c}p^m}\) when there is no confusion. 

For each \(m\geq1\), consider the divisor group
\[D_m := \Div(X_m^{(K)}) = \bb{Z}[X_m^{(K)}],\]
for which, thanks to Jacquet--Langlans correspondence, there is a natural action of the Hecke algebra \(\bb{T}_m\) (given by Hecke correspondences) and of the Galois group \(\mcal{G}_0\) (extended additively). We enhance these modules by considering the ordinary part, the branch of the family, the twist by the critical character and the inverse limit for all the members of the tower. This is:
\begin{align*}
    D_m^{\t{ord}}&:= D_m \otimes_{\bb{T}_m} \bb{T}_m^{\t{}{ord}},\\
    D_\infty^{\t{ord}}&:= \varprojlim_{m} D_m^{\t{ord}},\\
    \bf{D}_m&:=  D_m^{\t{ord}} \otimes_{\bb{T}_\infty^\t{ord}} \mcal{R},\\
    \bf{D}_m^\dag&:=  \bf{D}_m \otimes_\mcal{R} \mcal{R}^\dag,\\
    \bf{D}&:=  D_\infty^{\t{ord}} \otimes_{\bb{T}_\infty^\t{ord}} \mcal{R},\\
    \bf{D}^\dag&:=  \bf{D}_\infty^{\t{ord}} \otimes_\mcal{R} \mcal{R}^\dag.
\end{align*}
For all of them there is a well defined action of \(\mcal{G}_0\).\\

Denote by \(\bf{P}_{\frak{c},m}\) the image in \(\textbf{D}_m^\dag\) of the Heegner point \(P_{\frak{c},m} \in X_m^{(K)}\).
\begin{proposition}
    \(\textbf{P}_{\frak{c},m} \in H^0(\mcal{G}_{m},\textbf{D}_m^\dag)\).
\end{proposition}

\begin{proof}
    The definition of the twist in \(\bf{D}_m^\dagger\) fits perfectly the Galois compatibility property satisfied by our Heegner points (Proposition \ref{Compatibilidad}). This gives us
    \[\Theta(\sigma) \bf{P}_{\frak{c},m} = [\vartheta(\sigma)]\bf{P}_{\frak{c},m} = \<{\vartheta(\sigma)} \bf{P}_{\frak{c},m}.\]
    First equality holds because we are taking \(\sigma\in \mcal{G}_m\) and the square root makes sense.
\end{proof}

Since \(\bf{P}_{\frak{c}p^n,m}\) is a Heegner point of conductor \(\frak{c}p^{m+n}\) (not \(\frak{c}p^{n}\)), we apply corestriction maps.

Let \(\mathfrak{s},\mathfrak{t}\subset\mathcal{O}_F\) be ideals and \(Q\in H^0(\mathcal{G}_{\mathfrak{st}},\mathbf{D}_m^\dagger)\).
Consider the \(\Theta\)-twisted corestriction map
\begin{align*}
\t{cor}_{H_{\frak{st}}/H_{\frak{s}}}:\;
H^0(\mcal{G}_{\frak{st}},\bf{D}_m^\dagger)&\longrightarrow
H^0(\mcal{G}_{\frak{s}},\bf{D}_m^\dagger)\\
Q &\longmapsto \sum_{\eta\in\Gal(H_{\frak{st}}/H_{\frak{s}})}
\Theta(\eta^{-1})\, Q^{\eta}
\end{align*}

Define the elements \(\mcal{P}_{\frak{c},m}\) as follows:
    \[\mcal{P}_{\frak{c},m} = \t{cor}_{H_{\frak{c}p^m}/H_\frak{c}}(\textbf{P}_{\frak{c},m}) \in H^0(\mcal{G}_\frak{c},\textbf{D}_m^\dagger).\]

The sequence \(\{\mcal{P}_{c,m}\}_m\) satisfy the vertical compatibility condition:
\begin{proposition}
    \(\tilde{\alpha}_m(\mcal{P}_{\frak{c},m}) = U_p (\mcal{P}_{\frak{c},m-1})\).
\end{proposition}

\begin{proof}
    This is a consequence of Proposition \ref{Compatibilidad}. Check \cite[Proposition 7.2]{LV10} for more details. 
\end{proof}

This vertical compatibility relation of the points \(\{\mcal{P}_{\frak{c},m}\}_{\frak{c},m}\) lets us collapse them into the tower \(\bf{D}^\dag = \varprojlim_m \bf{D}_m^\dag\). Here it is crucial the ordinary invertibility of \(U_p\) on the Hida family.
\begin{definition}[Big Heegner point]
    We call \textbf{big Heegner point} of conductor \(\frak{c}\) to the element
    \[ \mcal{P}_\frak{c} := U_p^{-m}(\mcal{P}_{\frak{c},m}) \in H^0(\mcal{G}_\frak{c},\textbf{D}^\dagger).\]
\end{definition}

\begin{theorem} \label{Main theorem A}
    The following Euler system relations are satisfied, for each \(v\nmid \frak{nc}p^m\):
    \begin{align*}
        U_p(\mcal{P}_\frak{c}) &= \t{cor}_{H_{\frak{c}p}/H_\frak{c}}(\mcal{P}_{\frak{c}p})\\
        T_v(\mcal{P}_\frak{c}) &= \t{cor}_{H_{\frak{c}v}/H_\frak{c}}(\mcal{P}_{\frak{c}v})
    \end{align*}
\end{theorem}

\begin{proof}
    Both are the translation to \(\bf{D}^\dagger\) of Proposition \cite[Proposition 7.2]{LV10}. 
\end{proof}

\section{Definite setting: LV two-variable \(p\)-adic \(L\)-function}
During this section, \(B/F\) is a totally definite quaternion algebra and \(X_m\) denotes a Gross curve. Fix a Hilbert modular form of weight \(k\) as in the introduction and recall the ring \(\mcal{R}\) attached to the Hida family going through it. 

\begin{remark}
    Let \(\omega\) denote the (common) root number attached to a Hida family. By a result of Greenberg, \(\omega\) agrees with the root number of all but finitely many arithmetic specializations. Throughout our applications to \(L\)-values, we assume \(\omega=1\); that is, the common root number of the family coincides with the root number of the Rankin--Selberg \(L\)-function of \(f\) (which equals 1 since we work in the definite setting).
\end{remark}

Recall from Section \ref{Picard group and Jacobian} the module \(\bf{J}\). The following multiplicity one hypothesis, used in \cite{LV10} and later \cite{CL16}, is imposed for the rest of the work.
\begin{assumption}\label{Multiplicity one}
    Assume \(\bf{J}/\frak{m} \bf{J}\) has \(\mcal{R}/\frak{m}\)-dimension 1, where \(\mcal{M}\) is the maximal ideal of \(\mcal{R}\). 
\end{assumption}

\begin{proposition}\label{Isomorphism R and J}
    There is an isomorphism of \(\mcal{R}\)-modules between \(\mcal{R}\) and \(\bf{J}\).
\end{proposition}

\begin{proof}
    The proof over \(\bb{Q}\) can be found in \cite[Section 9.1]{LV10}. It is valid, word by word, in the totally real setting. We sketch the main steps of the proof. First, one sees that \(J_\infty^\t{ord}\) is a finitely generated \(\bb{T}_\infty^\t{ord}\)-module. This follows from the work of Hida (\cite[Section 1.9]{Hid94}), by comparing \(H_0(X_m(\bb{C}),\mcal{O}_f)\) (singular homology) and \(J_m\) for each \(m\). This implies that \(\bf{J}\) is a finitely generated \(\mcal{R}\)-module. The fact that it has rank one follows from Assumption \ref{Multiplicity one} and Nakayama's lemma. 
\end{proof}

For each conductor \(\frak{c}\), consider the map
\[\eta_\frak{c}: H^0(\Gal(K^{\t{ab}}/H_\frak{c}), \bf{D}^\dagger) \longrightarrow \bf{D} \longtwoheadrightarrow  \bf{J} \overset{\sim}{\longrightarrow} \mcal{R},\]
where the last map is a fixed isomorphism provided by Proposition \ref{Isomorphism R and J}. This morphism sends our big Heegner points to special points of \(\mcal{R}\), which contain relevant arithmetic information about the Hida family of \(f\).

\subsection{Two-variable \(p\)-adic \(L\)-function}
We construct a theta element in the sense of Bertolini--Darmon varying in Hida families. Using it, we define a \(p\)-adic L-function which is conjectured to interpolate special values of \(L\)-functions of the different specializations, including higher weight.\\

Let \(K_\infty/K\) be the \(p\)-anticyclotomic \(\bb{Z}_p\)-extension of \(K\). This is, the unique subfield \(K_\infty\) of 
\[H_\infty:= \varprojlim_n H_{p^n}\] 
satisfying, for \(d=[F:\bb{Q}]\),
\[\Gal(K_\infty/K) \simeq \bb{Z}_p^d.\]

Denote the intermediate layers \(K_n\) and the Galois groups \(G_n := \Gal(K_n/K)\). Define \(G_\infty:= \Gal(K_\infty/K)\). Denote, for the domain \(\mcal{R}\) attached to the Hida family, the group algebra
\[ \mcal{R}_\infty := \varprojlim_n \mcal{R}[G_n] = \mcal{R}[[G_\infty]].\] 
Consider, for \(n\in\bb{N}\), the integer
\[d(n):= \min\{t\in\bb{N}: K_n\subset H_{p^t}\}.\]
For \(\frak{c}=(1)\), recall the Heegner points \(\{\mcal{P}_{p^n}\}_{n\geq 1}\) of Section \ref{Construction of big Heegner points}. Since we are interested in the \(p\)-anticyclotomic \(\bb{Z}_p\)-extension, we consider the divisors
\[ \mcal{Q}_n:=\t{cor}_{H_{p^{d(n)}}/K_n}(\mcal{P}_{p^{d(n)}})\in H^0(\Gal(K^\t{ab}/K_n),\bf{D}^\dagger). \]

The construction of a \emph{big theta element} (theta elements suitable for Hida families) was first proposed in the rational setting by Longo--Vigni (\cite[Section 9.3]{LV10}). We extend this construction for families of Hilbert modular forms of parallel weight.\\

For each \(n\geq 1 \), define
\[ \theta_n:= U_p^{-n} \LL( \sum_{\sigma\in G_n} \eta_{K_n}(\mcal{Q}_n^\sigma)\otimes\sigma^{-1} \RR) \in \mcal{R}[G_n] \]
and consider 
\[\theta_\infty:= \varprojlim_n \theta_n \in \mcal{R}_\infty,\]
where the inverse limit is taken with respect to the natural morphisms \(G_m\rightarrow G_n\), \(m>n\), and their respective extensions \(\mcal{R}[G_m]\rightarrow \mcal{R}[G_n]\). The elements \(\theta_n\) are compatible under the corestriction maps, a consequence of the norm-compatibility of the big Heegner points \(\{\mcal{P}_\frak{c}\}_\frak{c}\).
 
\begin{definition}
    The \textbf{Longo--Vigni two-variable \(p\)-adic \(L\)-function} attached to the Hida family passing through \(f\) is
    \[\mcal{L}_p(f/K) := \theta_\infty \cdot \theta_\infty^* \in \mcal{R}_\infty,\]
    where \(*\) is the canonical involution of \(\mcal{R}_\infty\) induced by \(\sigma\mapsto\sigma^{-1}\) on group-like elements and acting trivially on \(\mcal{R}\).
\end{definition}

Choose \(\chi: G_\infty \rightarrow \mcal{O}^\times\) to be a finite character, hence factoring through some \(G_n\). Define
\[\mcal{L}_p(f/K;\chi) := \chi\LL(\mcal{L}_p(f/K)\RR) \in \mcal{R}.\]

One can use specialization maps of Hida theory for arithmetic primes \(\frak{p}\in \t{Spec}(\mcal{R})\) (that is, the kernel of an arithmetic map) using
\[\nu_\frak{p} : \mcal{R} \longrightarrow \mcal{R}_\frak{p}/\frak{p} \]
in order to define
\[\mcal{L}_p(f/K;\chi,\frak{p}):= \nu_{\frak{p}}(\mcal{L}_p(f/K,\chi)) \in \overline{\bb{Q}}_p.\]

The weight 2 specializations are the natural extension to the totally real setting of Bertolini--Darmon theta elements for the \(p\)-anticyclotomic \(\bb{Z}_p\)-extension. On the other hand, the higher weight specializations are conjectures to be strictly connected to the higher theta elements by Chida--Hsieh in \cite{CH16}.\\

The function \(\mcal{L}_p(f/K)\) is conjectured to interpolate the algebraic parts of the central \(L\)-values
\(L(f_\frak{p}^\dagger / K, \chi, k_\frak{p}/2)\) as \(\frak{p}\) varies. The following weaker conjecture, extending \cite[Conjecture 9.14]{LV10}, summarizes these ideas. 
\begin{conjecture}
    \(\LL. L(f_\frak{p}^\dagger/K,\chi,s) \RR|_{s=k_\frak{p}/2} \neq 0 \Longleftrightarrow \mcal{L}_p(f/K;\chi,\frak{p})\neq 0\).
\end{conjecture}
Here \(L(f_\frak{p}^\dagger/K,\chi,s)\) denotes the Rankin--Selberg L-function of the character \(\chi\) and the weight \(k_\frak{p}\) \(\frak{p}\)-specialization of the family, twisted by the critical character \(\theta_\frak{p}^{-1}\).\\

For a proof over \(\bb{Q}\) and a precise formula relating the two values for general weights, see \cite{CL16}.  

\subsection{Anticyclotomic Iwasawa main conjecture in the totally definite setting}
\label{Anticyclotomic Iwasawa main conjecture in the totally definite setting}

In the previous subsection we discussed the analytic side of the theory. We now review the algebraic counterpart and formulate an Iwasawa-type conjecture in the totally definite setting, relating the two. Our approach closely follows \cite[Section~9.3]{LV10}, and the conjecture below extends their Conjecture~9.12.

As in \cite{LV10}, the following assumption is essential.

\begin{assumption}
    The ring \(\mcal{R}\) is regular.
\end{assumption}

For an \(\mcal{R}_\infty\)-module \(M\), we adopt the following notation:
\begin{itemize}
    \item \(M^\vee := \operatorname{Hom}_{\bb{Z}_p}(M,\bb{Q}_p/\bb{Z}_p)\);
    \item \(M_{\mathrm{tors}}\) denotes the torsion submodule of \(M\);
    \item \(\operatorname{Char}_{\mcal{R}_\infty}(M)\) is the characteristic ideal of \(M\), defined by
    \[
        \operatorname{Char}_{\mcal{R}_\infty}(M) :=
        \begin{cases}
            \displaystyle\prod_{\operatorname{ht}(\frak{q})=1} \frak{q}^{\operatorname{length}(M_\frak{q})}, & \text{if } M = M_{\mathrm{tors}}, \\[1.2ex]
            0, & \text{otherwise}.
        \end{cases}
    \]
\end{itemize}

With this notation, consider the \(\mcal{R}_\infty\)-module
\[
    \tilde{H}^1_{f,\mathrm{Iw}}(K_\infty, \mathbf{A}^\dagger)
    := \varinjlim_n \tilde{H}^1_f(K_n, \mathbf{A}^\dagger),
\]
where the limit is taken with respect to the natural restriction maps, and
\[
    \mathbf{A}^\dagger := \operatorname{Hom}_{\bb{Z}_p}(\mathbf{T}^\dagger, \mu_{p^\infty}).
\]

\begin{conjecture}[Anticyclotomic Iwasawa main conjecture in the totally definite setting]
    There is an equality of nontrivial ideals in \(\mcal{R}_\infty\)
    \[
        \left( \mcal{L}_p(f_\infty / K) \right)
        = \operatorname{Char}_{\mcal{R}_\infty}
        \left( \tilde{H}^1_{f,\mathrm{Iw}}(K_\infty, \mathbf{A}^\dagger)^\vee \right).
    \]
\end{conjecture}

\section{Indefinite setting: Big Heegner classes \textit{à la} Howard}
The indefinite setting has been studied by different authors using different techniques. See \cite{Fou12} and \cite{Dis22} over totally real fields. In the rational non-ordinary case, see \cite{JLZ20} and \cite{Roc25}. In this paper, we extend the construction of Howard of a big Kummer map, which exploits the independence of the weight in Hida theory to build a compatible family of cohomology classes starting only from classical weight 2 objects. The comparison between specializations of these big Heegner classes and the classes obtained using generalized Heegner cycles remains an open question. The work of Castella (\cite{Cas19}) and Ota (\cite{Ota20}), and recently in the indefinite quaternionic setting in \cite{LMW25}, might be extended to the totally real setting.\\

During this section, we work with \(B\) a quaternion algebra over \(F\), ramified at all but one places over \(\infty\). Denote this place \(\tau\). The Heegner points \([(f,g)]\) defined in \(X_m^{(K)}\) will correspond to CM points in the totally real Shimura curve \(X_m\) studied in Section \ref{Gross curves and Shimura curves over totally real  fields}. In fact, these points will be rational over the field \(H_m[m]\).\\

 Let \(H_\frak{c}^{(\frak{n}p)}\) denote the maximal abelian extension of \(H_\frak{c}\) unramified outside \(\frak{n}p\) and 
\[G_\frak{c}^{(\frak{n}p)}:=\Gal(H_\frak{c}^{(\frak{n}p)}/H_\frak{c}).\]

We consider a twisted Kummer map, that is, a version of the classical Kummer map for elliptic curves, but for the Tate module twisted by \(\Theta^{-1}\), as we have been doing along this work. Define
\[\delta_m: H^0(H_\frak{c}, \textbf{J}_m^\dag (H_m[m]))\longrightarrow H^1(G_\frak{c}^{(\frak{n}p)}, \bf{T}_m^\dag).\]
Recall
\[\underline{\t{Ta}}_{p,m}^\t{ord}:= e^\t{ord}\cdot \t{Ta}_p(\t{Jac}(X_{0,1}(\frak{n},p^m)))\otimes_{\bb{Z}_p}\mcal{O}\]
and consider
\[\underline{\bf{T}}^\dagger_m:= \underline{\t{Ta}}_{p,m}^\t{ord} \otimes _\mcal{R}  \mcal{R}^\dagger.\]

The definition of the map \(\delta_m\) is as follows. For \(P\in \bf{J}_m^\dagger(H_m[m])\) fixed by the action of \(\Gal(H_m[m],H_\frak{c})\), choose \(Q_n \in \bf{J}^\dagger_m(\overline{H_\frak{c}})\) such that \(p^n Q_n = P\). Attach to it the 1-cocycle defined by
\[ b_n(\sigma) = \Theta^{-1}(\sigma) Q_n^\sigma - Q_n \in \bf{J}_m^\dagger[p^n].\]
Set then
\[ \delta_m(P) = \varprojlim_n\ b_n \in H^1(G_\frak{c}^{(\frak{n}p)},\underline{\bf{T}}^\dagger_m).\]
\vspace{3mm}

Define
\[\delta_\infty: H^0(G_\frak{c}^{(\frak{n}p)},\bf{J}^\dagger(H_{\frak{c}p^\infty}[\infty]) ) \longrightarrow H^1(G_\frak{c}^{(\frak{n}p)},\bf{T}^\dagger)\]
by considering inverse limits and identifying \(\underline{\bf{T}}^\dagger\) with \(\bf{T}^\dagger\) as a consequence of Jacquet--Langlands correspondence using Proposition \ref{Comparacion representaciones}.\\

Using the Heegner points of Section \ref{Construction of big Heegner points} we define big Heegner classes.
\begin{definition}[Big Heegner class]
    A \textbf{big Heegner class} of conductor \(\frak{c}\) is the element
    \[\delta_\infty(\mcal{P}_c) = \kappa_\frak{c}\in H^1(G_\frak{c}^{(\frak{n}p)},\textbf{T}^\dag).\]
\end{definition}

The relation properties between the big Heegner points \(\{\mcal{P}_\frak{c}\}_\frak{c}\) are directly translated to big Heegner classes \(\{\kappa_\frak{c}\}_\frak{c}\) as summarized in the next statement.
\begin{theorem}\label{Main Theorem 2}
The following Euler system relations are satisfied
\begin{align*}
    U_p(\kappa_\frak{c})&= \t{cor}_{H_{\frak{c}p}/H_\frak{c}}(\kappa_{\frak{c}p}),\\
    T_v(\kappa_\frak{c})&= \t{cor}_{H_{\frak{c}v}/H_\frak{c}}(\kappa_{\frak{c}v}),
\end{align*}
for \(v\nmid \frak{nc}p\) inert in \(K\).
\end{theorem}

\begin{remark}
    This construction coincides with the one of Longo and Vigni for \(F=\bb{Q}\) and with Howard if additionally \(\frak{n}^{-}=(1)\). As happens over \(\bb{Q}\), the relation between these classes and the ones obtained by \cite{Fou12} or \cite{Dis22} has not been studied yet.
\end{remark}

\end{document}